\documentclass[11pt, a4paper]{amsart}
\usepackage{amssymb,amsmath,amscd,amsfonts,mathrsfs,yfonts}
\usepackage{textcomp}
\usepackage{appendix}
\usepackage{upgreek}
\usepackage[matrix,arrow,curve,pic,2cell]{xy}
\UseTwocells
\xyoption{2cell}
\usepackage{tkz-graph}
\usetikzlibrary{decorations.markings}

\usetikzlibrary{shapes.geometric}
\newtheorem{definition}{Definition}
\newtheorem{theorem}{Theorem}
\newtheorem{lemma}{Lemma}
\newtheorem{remark}{Remark}
\newtheorem{corollary}{Corollary}
\DeclareMathOperator{\Hom}{Hom}
\DeclareMathOperator{\Dom}{Dom}
\title{A Lipschitz version of de Rham theorem for $L_p$-cohomology}
 \author[V. Gol{\textquotesingle}dshtain]{Vladimir Gol{\textquotesingle}dshtain}
 \address{V. Gol{\textquotesingle}dshtain: Department of Mathematics, Ben Gurion University of the Negev, P. O. Box 653, Beer Sheva, Israel}
 \email{vladimir@bgu.ac.il}
\author[R. Panenko]{Roman Panenko} 
\address{R. Panenko: Department of Mathematics, Ben Gurion University of the Negev, P. O. Box 653, Beer Sheva, Israel}
\email{panenkora@gmail.com}
\usepackage[utf8]{inputenc}
\begin{document}
\begin{abstract}
We focus our attention on the de Rham operators' underlying properties which are specified by  intrinsic  effects of  differential geometry structures. 
And then we apply the procedure of regularization in the context of Lipschitz version of de Rham calculus  on metric  simplicial complexes with bounded geometry. 

\vspace{2mm}

\vspace{2mm}
\noindent
\textbf{Mathematic Subject Classification 2000: } 58A12, 58A15, 53C65, 57Q05, 57R12, 51F30, 46E30
\end{abstract}
\maketitle

\section{Introduction}
The reasons which lie at the roots of the present text could be abstracted as follows.  
Despite the fact that the notion of de Rham's regularization operators has a long history and some useful applications, primarily, it was 
a tool that allows to reduce cohomology of Banach chain complexes to the case which is more familiar and convenient  since it is presented by a subcomplex of  smooth objects.
It does not leave the impression of clarity.  
Indeed,  de Rham's initial exposition on the subject and further applications, focusing on the analytic aspect of the matter, seem to tend cloudy intrinsic elegance and simplicity of that construction.  Such situation inherently encourages us to reopen the discourse on the subject in order 
to embellish the prevailing approach  and see how far that construction could be generalized.     

Beginning with the first decades of  the 20th century when the basic notions of the exterior calculus were formulated thanks to \'{E}lie Cartan's works 
and further through  Georges de Rham's contribution one got the perfectly clear language to talk about global properties of manifolds. 
In particular, his explorations led up to emergence of the concept of so-called de Rham's complex, namely,  that work elucidated the analogies between 
differential forms and chains. One can notice that the concept of chain complex was quite known within the frames of algebraic topology and homological algebra that was being formed at that time.     
Also  we should attribute to  that period (around the thirties)  the de Rham's theorem establishing an isomorphism between the cohomogy of  differential forms  and the singular  cohomology.
Eventually, those reasons led up to the notion of current generalizing essential characteristics shared by both chains and forms. Later de Rham extensively developed  
the theory of currents involving as the foundation of Lauren Schwartz's work on distributions.   
That yielded subsequently the results on approximation of currents by  smooth forms  and required to introduce the regularization operators defined in the weak sense. 
It was regarded as relying on a duality between currents and compactly supported smooth forms. 
Let us take a closer look at the subject.  

There is the well-known idea to approximate a value of locally integrable function $f$ at every point with its mean value over a bounded neighbourhood of the such point. 
And more generally, using a convolution with a smooth kernel  $\varphi$    such that 
$$
\int_{\mathbb R^n} \varphi(x)dx  = 1,
$$
we can define a regularization operator 
$$
\Phi_{\varepsilon}(f) = f \ast \varphi_{\varepsilon}.
$$
Let $\mathscr D^{\prime}(\mathbb R^n)$ be a  space of continuous functionals on the space $C^{\infty}_{0}(\mathbb R^n)$ endowed with the usual topology.
The convolution $T\ast \varphi$ of a distribution $T \in \mathscr D^{\prime}(\mathbb R^n)$ and a function $\varphi(y) \in C^{\infty}_{0}(\mathbb R^n)$ is 
defined by 
$$
\{T\ast \varphi\}(x) = T(\varphi(x-y))
$$
that is 
the operator $T \mapsto \langle T,\, \tau_x \tilde{\varphi} \rangle$ where $\tilde{\varphi}(y) = \varphi(-y)$, $ \tau_x\varphi(y) =  \varphi(y+x)$. 

We can sum up that approach in the following way, see for example \cite{WR}.
\begin{theorem}
Let  $\varphi_{\varepsilon} \in C^{\infty}_{0}(\mathbb R^n)$ be a sequence of positive functions such that
$$
\int_{\mathbb R^n} \varphi_{\varepsilon}(x)dx  = 1
$$
and $\rm{supp}(\varphi_{\epsilon})$ is a ball with radius $\varepsilon$.
If  $T \in \mathscr D^{\prime}(\mathbb R^n)$  it follows that
$T\ast \varphi_{\varepsilon}  \in C^{\infty}(\mathbb R^n)$ and that  $T\ast \varphi_{\varepsilon} \to T$ in $\mathscr D^{\prime}(\mathbb R^n)$ as $\varepsilon \to 0$.
\end{theorem}
In the light of the above it is reasonable to talk about the regularization of  distribution in the following sense 
$$
\Phi_{\varepsilon}(T) = T \ast \varphi_{\varepsilon}.
$$
If $\varphi_{\varepsilon}$ is a symmetric kernel, that is $\varphi_{\varepsilon}(x) = \varphi_{\varepsilon}(-x)$, and $g  \in C^{\infty}_{0}(\mathbb R^n)$ 
we have
$$
\langle T\ast \varphi_{\varepsilon},\, g\rangle = \langle T,\, g \ast \varphi_{\varepsilon}\rangle. 
$$
Then we can define a regularization operator with a symmetric kernel on the space of distributions as follows
$$
\{\Phi_{\varepsilon}(T)\} (g) = T(\Phi_{\varepsilon}(g)).
$$

Let $M$ be a differentiable manifold and  let $(\Omega^{*}(M),\,d)$ 
be the de Rham $DG$-algebra (differential draded algebra) on $M$, that is the algebra of smooth differential forms. In particular, there is defined a chain complex
$$
\xymatrix{0\ar[r]&C^{\infty}(M) \ar[r] &\Omega^{1}(M)\ar[r]&\dots\ar[r]&\Omega^{n}(M)\ar[r]&0}.
$$
Following the de Rham approach we turn to the subcomplex of compactly supported forms and  its dual complex of currents.
Consistent with the above, we intend to define a regularization operator on  currents analogously to the case of distributions:
$$
R T[\omega] = T[R^*\omega].
$$
In line with what was said above there emerges a reasonable question how we should think of a procedure of computing the mean value of a differential form   $R^*\omega$.
It is quite clear that we need to define the operator under consideration in such a manner that preserves cohomology classes. 

First of all we should clarify the notion of homotopy.
Let $\mathsf{A}$ be an additive category. Consider the category of chain complexes $\mathsf{Ch(A)}$.  We can introduce the homotopy category of chain complexes $\mathsf{K(A)}$ taking into account 
a concept of `equivalent deformation' $\eta$ of morphisms ${f,\,g\in{\Hom}_{\mathsf{Ch(A)}}(V,\,W)}$
$$
\xymatrix{V \rtwocell<5>^{<1.5>f}_{<1.5>g}{\  \eta}& W}.
$$
Take a look at this construction in more detail. The homotopy $\eta$ we assume  to be a set of morphisms $\{\eta^i \in {\Hom}_{\mathsf{A}}(V^i,\,W^{i-1})\}$ which satisfy 
$$
f^i-g^i = d_{W}^{i-1}\eta^i + \eta^{i+1}d_{V}^{i}.
$$
It should be noted that we do not involve the requirement that $\eta$ is a chain morphism.  The condition of existence of such a homotopy between morphisms equips  the set  
${\Hom}_{\mathsf{Ch(A)}}(V,\,W)$ with an equivalence relation. And so $\mathsf{K(A)}$ can de introduced as a category of chain complexes with morphisms defined modulo homotopy. 
We can reveal the point by turning to the well-studied case of Abelian categories which are the classical setting for the treatment of homological algebra. It is not hard to see that homotopic morphisms 
induce the same morphism between the corresponding cohomology groups and every homotopy equivalence $f\colon V\to W$ defines the isomorphism of cohomologies.  
Thus
a two-sided invertible morphism in the category ${\sf K(A)}$ 
corresponds to an isomorphism of cohomologies. In particular, a homotopy equivalence between topological spaces induces isomorphism between the singular chain complexes 
in  ${\sf K(A)}$. That could be extracted as the essence of  the Poincar\'{e} lemma.

Another example of chain homotopy will serve as the central part of our interpretation of the regularization operators.
Assume that $X$ is a vector field, then Lie derivative ${\mathcal L}_X$ 
is the $0$-derivation of the $DG$-algebra 
such that there exists a $\text{-}1$-derivation $\iota_X$ being the homotopy between ${\mathcal L}_X$ and the zero map
$$
\xymatrix{(\Omega^{*}(M),\,d) \rrtwocell<8>^ {<1.5>0}_{<1.5>{\mathcal L}_X}{\ \ \iota_X}& &(\Omega^{*}(M),\,d)}.
$$
That is to say,  Lie derivative  satisfies Cartan's formula:
$$
{\mathcal L}_X = d \iota_X + \iota_X d.
$$
At every point $x$ the map  
$
\omega \mapsto {\mathcal L}_X \omega
$
induces a function with values in the exterior power of the cotangent space at $x$
$$
{\mathcal L}_X \omega_{x}  \colon \mathbb R \to \mathcal \bigwedge^n  T^*_x M
$$
 ${\mathcal L}_X \omega$ is precisely nothing else than the instantaneous velocity in the exterior power of the cotangent space, 
that is a magnitude of change of the form $ \omega$
under the infinitesimal translation along an integral curve $\phi_X(t)$.
That corresponds to the zero endomorphism 
of $(\Omega^{*}(M),\,d)$  
in the homotopy category of chain complexes. 
In other words, for every closed form $\omega$ the form ${\mathcal L}_X\omega$ is an exact form. 

We can compute the integral  of the function ${\mathcal L}_X \omega_{x}$. 
Owing to the linearity of integration and that fact that it commutes with the exterior differential 
we  can see that the following 
$$\int\limits_0^1 {\mathcal L}_X\omega dt = \phi^*_{t=1}\omega - \omega$$ 
implies the homotopy 
$$
\xymatrix{(\Omega^{*}(M),\,d) \rrtwocell<8>^ {<1.5>\phi^*_{t=1}}_{<1.5>\rm Id}{}& &(\Omega^{*}(M),\,d)}.
$$
That makes sense to talk about the procedure of regularization on differential forms. 
Namely the pullback of translation along vector fields preserves cohomology classes. And as a result,  combining 
the intrinsic attribute of smooth manifolds expressed in the Cartan's formula and the classic idea of mollifier 
we can define a form representing the mean value of the given differential form at every point of smooth manifold.   

Turning back to the de Rham's construction it is not hard to see that  the operator
$$
R T[\omega] = T[R^*\omega]
$$
defined on currents inherits the property to preserve cohomology classes 
\begin{align*}
RT[\omega] &= T[R^*\omega] = T[A^*d\omega + dA^*\omega] =  T[A^*d\omega] + T[dA^*\omega] \\
&= AT[d\omega]+ \partial T[A^*\omega] = \{\partial A T + A \partial T \}[\omega]. 
\end{align*}

The next step in that direction was made in the work \cite{GKS1}. Where authors focused on a special kind of currents which can be presented as elements
of Sobolev space of differential forms $\Omega_{p,\,p}^*(M)$. It is clear that being a special case of currents such forms hold all basic properties of the regularization. 
The crucial result consists in the proof that we have the same diagram
$$
\xymatrix{(\Omega_{p,\,p}^{*}(M),\,d) \rrtwocell<8>^ {<1.5>R}_{<1.5>\rm Id}{}& &(\Omega_{p,\,p}^{*}(M),\,d)}.
$$
in the category of chain complexes of Banach spaces. 
That allows us to generalize the de Rham's theorem to the case of $L_p$-cohomology of triangulable noncompact manifolds.

We call a simplicial complex $K$ having bounded geometry if every vertex of the $1$-skeleton of $K$ has a uniformly bounded degree as a vertex of graph 
and  the length of every edge is in the interval $[L^{-1},\,L]$ for some $L \ge 1$.  

We introduce a class of differential forms $S{\mathscr L}^{*}_p(K)$ on a simplicial complex $K$ 
which are locally images of smooth forms defined on a subsets of $\mathbb R^n$ under bi-Lipschitz homomorphisms and  
have a finite graph norm on the domain of  $S{\mathscr L}^{k}(K) \xrightarrow{d} S{\mathscr L}^{k+1}(K)$ in the sense of $L_p$-spaces. Let $\Omega^{*}_{p,\,p}({ K})$ denote the closure of that class under a topology induced from the graph norm.
The main result of the present  work can be summed up in two assertions:
\begin{itemize}
\item \emph{Let $K$ be a  complex of bounded geometry. Then there exists the diagram
$$
\xymatrix{ \dots\ar[r]^-{d}&\Omega^{k-1}_{p,\,p}({ K}) \ar[rr]^-{d}\ar[d]^-{\mathscr{R}}
&&\Omega^{k}_{p,\,p}({ K})\ar[rr]^-{d}\ar[lldd]_{\mathscr{A}}\ar[d]^-{\mathscr{R}}
&&\Omega^{k+1}_{p,\,p}({ K})\ar[lldd]_{\mathscr{A}}\ar[r]^-{d}\ar[d]^-{\mathscr{R}}&\dots\\
\dots &S{\mathscr L}^{k-1}_p (K)\ar@{_(->}[d]&&S{\mathscr L}^{k}_p(K)\ar@{_(->}[d]&&S{\mathscr L}^{k+1}_p(K)\ar@{_(->}[d]&\dots\\
\dots\ar[r]_-{d}&\Omega^{k-1}_{p,\,p}(K)\ar[rr]_-{d}&&\Omega^k_{p,\,p}(K)\ar[rr]_-{d}&& \Omega^{k+1}_{p,\,p}(K)\ar[r]_-{d}&\dots
}
$$
with commutative squares in the category of Banach spaces ${\sf Ban}_{\infty}$. Moreover, the map $\mathscr R$ is homotopic to the identity
$$
\xymatrix{(\Omega_{p,\,p}^{*}(M),\,d) \rrtwocell<8>^ {<1.5>\mathscr R}_{<1.5>\mathrm{Id}_{\Omega^*_{p,\,p}}}{\ \ \mathscr A}& &(\Omega_{p,\,p}^{*}(M),\,d)}.
$$}
\item \emph{Under the same conditions the following commutative triangle of isomorphisms takes place  in the category ${\sf Vec}_{\mathbb R}$
$$
\xymatrix@R = 1.5em @C = 0.1em{&H^k(S{\mathscr L}^*_p(K) )\arrow[rdd]^{\mathscr I}&\\
&&\\
H^k(\Omega^*_{p,p}(K))\arrow[ruu]^{\mathscr R}\arrow@{-->}[rr]&&H^k(C^*_p(K))}
$$}
\end{itemize}
It should notice that such complexes could be useful as a bi-Lipschitz triangulation of Riemannian manifolds with bounded geometry, see Appendix A.  
\newpage
\section{Homotopy and Lie Derivative}

Most of the content included in this section could be found in \cite{FW}.

Let $\mathscr C^{\infty}({\mathbb R}^n)$ be a ring of smooth functions on $\mathbb R^n$. 
The differentiation of smooth functions defines a derivation in $\mathscr C^{\infty}({\mathbb R}^n)$ with values in  $\mathscr C^{\infty}$-module consisting of $1$-forms on $\mathbb R^n$
$$
d \colon \mathscr C^{\infty}({\mathbb R}^n) \to  \Omega(\mathbb R^n),
$$
that is a homomorphism of the respective additive groups which satisfies the condition 
$$
d(fg) = fdg + gdf.
$$
As usual, we define all  operation in $\mathscr C^{\infty}({\mathbb R}^n)$ pointwisely. 
There are the well-known algebraic reasons which imply a number of facts about local structure of $\mathscr C^{\infty}({\mathbb R}^n)$.
Let $\mathscr{C}^{\infty}_x(\mathbb R^n)$ be the space of germs of smooth functions at a point $x$. 
Then $\mathscr{C}^{\infty}_x(\mathbb R^n)$ is a commutative local ring, that means that non-invertible elements, namely germs of functions which vanish at $x$, 
form a maximal ideal $\mathfrak m_x$. 
And in consequence the quotient ring $\mathscr{C}^{\infty}_x(\mathbb R^n)/\mathfrak m_x$ is a field. As a result we can conclude that 
$\mathfrak m_x/\mathfrak m_x^2$ is a vector space. Assume that $f \in \mathscr{C}^{\infty}_x(\mathbb R^n)$ then due to the Taylor's 
theorem we can write down the following 
$$
f - f(x) = \langle \nabla_x f,\, \sum_{i = 1}^n (\xi_i-\xi_i(x))\vec{e_i}\rangle + h,~h \in m_x^2.
$$
and as a result we obtain the representation of $f$ at the point $x$  as an element  of vector space   $\mathfrak m_x/\mathfrak m_x^2$
$$
f \mapsto f - f(x) \in \mathfrak m_x/\mathfrak m_x^2.
$$
To sumarize, fibers $x \mapsto \mathfrak m_x/\mathfrak m_x^2 $ specify a vector bundle.  Considering that components of the vector $f-f(x)$ change smoothly  on $\mathbb R^n$,
we obtain a vector field corresponding the element of  $\mathscr C^{\infty}({\mathbb R}^n)$. So we can sum up that  the procedure outlined above  allows us to define a derivation
 in the commutative ring  $\mathscr C^{\infty}({\mathbb R}^n)$ with values in  $\mathscr C^{\infty}$-module  consisting of  vector fields. As a consequence, 
 there exists a related derivation in $\mathscr C^{\infty}({\mathbb R}^n)$ with values in the dual $\mathscr C^{\infty}$-module consisting of $1$-forms
 $$
d \colon \mathscr C^{\infty}({\mathbb R}^n) \to  \Omega(\mathbb R^n).
$$
Now we can define  the de Rham $DG$-algebra on $\mathbb R^n$ as a graded algebra
$$
 \Omega^*({\mathbb R}^n) =  \bigoplus_{k=0}^n\bigwedge^k\Omega({\mathbb R}^n),~\bigwedge^0\Omega({\mathbb R}^n) =  \mathscr C^{\infty}({\mathbb R}^n)
$$
endowed with an antiderivation, i. e. an endomorphism satisfying the graded Leibnitz rule with a commutator factor $-1$,  which is specified as the exterior derivative $d$ in the usual sense: 
\begin{itemize}
\item   define the derivative in accordance with the  derivation in a ring \linebreak
${d \colon \mathscr C^{\infty}({\mathbb R}^n) \to  \Omega(\mathbb R^n)}$ for $0$-forms;
\item   $d^2 = 0$;
\item  $d(\omega\wedge\theta) = (d\omega)\wedge\theta + (-1)^{\deg\omega}\omega\wedge(d\theta)$.
\end{itemize}

Let $X$ be a vector field on $\mathbb R^n$ and  $\phi_X\colon {\mathbb R}^n \times {\mathbb R} \to {\mathbb R}^n$ be a corresponding flow. 

Consider a $\text{-}1$-antiderivation  on the graded algebra $\Omega^*({\mathbb R}^n)$ defined as the following.
\begin{definition}
Let $X$ be a vector field on $\mathbb R^n$. The interior product is a map 
$$
\iota_X \colon \Omega^n(\mathbb R^n) \to \Omega^{n-1}(\mathbb R^n)
$$
satisfying a number of conditions 
\begin{itemize}
\item[$\bullet$] If $\omega \in \Omega^1(\mathbb R^n)$ we put $\iota_X\omega = \langle\omega,\,X\rangle$, i. e. $\iota$ is the canonical pairing;
\item[$\bullet$] $\iota(\omega\wedge\theta) = (\iota_X\omega)\wedge\theta + (-1)^{\deg\omega}\omega\wedge(\iota_X\theta)$.
\end{itemize}

\end{definition}
\begin{definition}
Let $X$ be a vector field on $\mathbb R^n$. The Lie derivative is a map
$$
 {\mathcal L}_X \colon \Omega^n(\mathbb R^n) \to \Omega^{n}(\mathbb R^n)
$$
defined in the following way
$$
 {\mathcal L}_X\colon \omega \mapsto \frac{d}{dt}\bigg|_{t=0} \phi^*_X \omega.
$$
\end{definition}
\begin{theorem}Under the above assumptions the Cartan's magic formula holds: 
$$ {\mathcal L}_X\omega = \iota_X \circ d\omega + d \circ \iota_X \omega $$
\end{theorem}
It could be proven by following the well-known idea, namely, one only needs  to check out a number of simple assertions:
\begin{itemize}
\item[$\bullet$] ${\mathcal L}_X$ satisfies Leibniz's law;
\item[$\bullet$] $[{\mathcal L}_X,\,d] = 0$;
\item[$\bullet$] The anticommutator $\{\iota_X,\,d\}$ is a $0$-derivation of $DG$-algebra and \linebreak $[\{\iota_X,\,d\},\,d] = 0$ (almost evident);
\item[$\bullet$] Now it just remains to check by  induction that a couple of derivations of $DG$-algebra, which agree for dimensions $n=0,\,1$ and commute with $d$, are equal for all dimensions.
\end{itemize}
\begin{corollary}
The Poincar\'{e} lemma 
$$
H^i(\Omega^*(U) )=0 \text{~for~} i < n, \text{~where $U$ is an open ball in $\mathbb R^n$},
$$
can be derived from the Cartan's formula.
\end{corollary}
Let us consider the change of the form $\omega$ along a segment  \linebreak
${\phi_X(x,\,t)\colon [0,\,1]\to {\mathbb R}^n}$ of the integral curve which  starts at the point $x$.
A  parametrized differential  form ${\phi^{*}_{X}\omega(t) \in \Omega^{k}({\mathbb R}^n)}$ defines a family of multilinear skew-symmetric maps 
on the tangent  space $T_x\mathbb R^n$
for an integral curve that starts at the point $x$:
$$
f(t)A^0_x(t)\wedge \dots \wedge A^{k-1}_x(t)\colon \bigwedge^k{\mathbb R^n} \to {\mathbb R},~A^i_x(t) \in ({\mathbb R}^n)^*
$$
that is, there is specified function
$$
F(t) = f(t)\det (A^i_x(t)\xi_j)
$$
at every $\xi_0\wedge\dots\wedge \xi_{k-1}$, 
and so we can define $\frac{d}{dt} F(t)$ and $\int\limits^{1}_{0}  F'(t) dt$.

That induces a couple  of maps:
$$
{\mathcal L}_{X}\omega(t) \colon \bigwedge^k{\mathbb R^n} \to {\mathbb R},~\text{at every point $t$}
$$
where 
$$
\{{\mathcal L}_{X}\omega(t)\}(x) = {\mathcal L}_{X(\phi_t(x))}\{\phi_X^*(t)\}\omega
$$
and
$$
\int\limits^{1}_{0} {\mathcal L}_{X}\omega(t) dt \colon \bigwedge^k{\mathbb R^n} \to {\mathbb R}.
$$

Then we have 
\begin{align*}
\int\limits_0^1 {\mathcal L}_{X}\omega dt =& \phi_X^*\big|_{t=1}\omega - \omega \\
= \int\limits_0^1 \iota_{X(\phi_t(x))} \circ \{\phi^*_X(t)\} (d \omega ) dt
&+ d  \bigg(\int\limits_0^1 \iota_{X(\phi_t(x))}  \circ \{\phi^*_X(t)\} (\omega)  dt\bigg)
\end{align*}

\begin{definition} 
Let $v \in {\mathbb R}^n$. Define an associated flow:
$$
s_v\colon {\mathbb R}^n\times {\mathbb R} \to {\mathbb R}^n 
$$
as the translation along $v$:
$$
s_{tv}(x) = x+tv.
$$
\end{definition}
 Then we can use  Cartan's  formula 
 $$
  s^*_v\omega -  \omega = Q_{v} d\omega + d Q_{v} \omega
  $$
 where
  $$
  Q_{v} = \left\{ \int\limits^{1}_{0} dt\right\} \circ \iota_{v} \circ \phi_X^*(t).
 $$
 \section{de Rham operators on $\Omega^{*}({\mathbb R}^n)$}
Let $f\colon {\mathbb R}^n \to {\mathbb R}$ be a compactly supported smooth function such that \linebreak
${ {\rm supp}(f) \subset B_1}$, 
${\int_{{\bf R}^n}}f(v)dv^0 \dots dv^{n-1} = 1$, $f(v)\ge 0$ and ${f(v) = f(-v)}$. Let us put $\tau(v) = f(v)dv^0 \dots dv^{n-1}$.
Using the previous argumentation, we can integrate the equation 
$$
s^*_v \omega \cdot \tau(v)- \omega \cdot \tau(v) = d Q_v(\omega) \cdot \tau(v)+ Q_v(d\omega) \cdot \tau(v)
$$
at every point $x$, namely this integration procedure is induced by integration of forms of the following type  $g(v)\det(\langle A^i(v),\, * \rangle)\tau$, 
where \linebreak ${\omega(x) = g(v)A^0(v)\wedge \dots \wedge A^k(v)(*\wedge \dots \wedge *)}$:
$$
\int\limits_{{\mathbb R}^n}\left(s^*_v \omega(x) \cdot \tau(v)- \omega(x) \cdot \tau(v)\right) = \int\limits_{{\mathbb R}^n}\left(d Q_v(\omega(x)) \cdot \tau(v)+ Q_v(d\omega(x)) \cdot \tau(v)\right),
$$
that is what allows us to specify a chain homotopy:
$$
\mathcal{R}( \omega) - \omega = d \mathcal{A(\omega)} + \mathcal{A}( d\omega).
$$

There is a diffeomorphism $h$ of ${\mathbb R}^n$ onto the open ball ${\bf B}_1$ with centre $0$ and radius $1$. Let $U\subset {\mathbb R}^n$  and ${\bf B}_1 \subset U$.
We can define $\mathfrak s_v\colon U \to U$
$$
\mathfrak s_v x =\begin{cases}
hs_vh^{-1}(x),~ \text{if $x \in {\bf B}_1$};\\
x,~ \text{if $x \notin {\bf B}_1$}.
\end{cases}
$$

It was shown in de Rham's book \cite{dR} that  $\mathfrak s^*_{tv}$ produces a group action of the additive group of real numbers on $U$, that is, 
$\mathfrak s^*_{(t_0+t_1)v} = \mathfrak s^*_{t_1v}\circ \mathfrak s^*_{t_0v}$.
And also we can say that  ${{\frak X}_v = d_{h^{-1}(x)}h(v)}$ is a vector field  consisted of tangent vectors to $\mathfrak s_{tv} (x)$, then 
we have 
$$
\frac{d}{dt}\Biggl|_{t=0}\mathfrak s^*_{tv}  \omega =  d\circ \iota_{\frak X_v}(\omega) + \iota_{\frak X_v} \circ d(\omega);
$$
and
$$
\mathcal{R}_{\varepsilon}( \omega) - \omega = d \mathcal{A}_{\varepsilon}(\omega) + \mathcal{A}_{\varepsilon}( d\omega).
$$
where
$$
\mathcal{R}_{\varepsilon} \omega= \int\limits_{{\mathbb R}^n}\mathfrak s^{*}_{\varepsilon v} \omega(x) \cdot \tau(v),~
\mathcal{A}_{\varepsilon}(\omega)  =  \int\limits_{{\mathbb R}^n} \left(\int\limits^{1}_{0} \iota_{\frak X_{\varepsilon v}(\mathfrak s_{\varepsilon vt}(x))}
(\mathfrak s^{*}_{\varepsilon vt}\omega) dt\right) \cdot \tau(v).
$$
It was shown \cite{GKS1} that 
\begin{lemma}
For every $\varepsilon > 0$  the maps $\mathcal{R}_{\varepsilon}$ and  $\mathcal{A}_{\varepsilon}$ are bounded on $\Omega_p^k({\bf B_1})$ with respect to the $L_p$-norm
 and moreover the following estimations hold
$$
\|\mathcal{R}_{\varepsilon}\|_p\le C(\varepsilon),
$$
$$
\|\mathcal{A}_{\varepsilon}\|_p\le M(\varepsilon);
$$
where $C(\varepsilon) \to 1$, $M(\varepsilon) \to 0$ as $\varepsilon \to 0$.
\end{lemma}
From now on, we will use ${\rm res}_{V,\,U}$ for the restriction map $\Omega^k(U) \to \Omega^k(V)$ where $V \subset U$ if it is well-defined in the context of consideration. 
\begin{lemma}
Let ${\bf B}_1$ be a closed ball in $\mathbb R^n$ with centre $0$ and radius $1$, ${\bf B}_1 \subset U\subset {\mathbb R}^n$.
Then for every $\varepsilon > 0$ and any compact $F \subset {\rm Int\,}{\bf B}_1$  the map ${\rm res}_{F,\,U} \circ \mathcal{R}_{\varepsilon}$ is a bounded 
operator $\Omega^{k}_{p}(U) \to \Omega^k_{\infty}(F)$. 
\end{lemma}
\begin{proof}
Let $\omega \in \Omega^{k}_{p}(U)$ then $|\mathcal{R}_{\varepsilon} \omega|$ is a smooth function that  implies  it is bounded and there is a point $\xi$ such that 
$\sup_{x \in F} |\mathcal{R}_{\varepsilon} \omega|(x) = |\mathcal{R}_{\varepsilon} \omega|(\xi)$.
$$
|\mathcal{R}_{\varepsilon} \omega|^p(\xi) \le \big (\int\limits_{{\rm supp}(f) } |\mathfrak s^{*}_{\varepsilon v} \omega|(\xi) \cdot \tau(v)\big)^p
$$
$$
\le C \int\limits_{{\rm supp}(f) } |\mathfrak s^{*}_{\varepsilon v} \omega|^p(\xi) dv^0 \dots dv^{n-1} \le C \|\omega\|^p_{\Omega^k_{p}(U)}
$$
where $C = {\rm mes}({\rm supp}(f) )^{p-1} (\sup_{x \in {\rm supp}(f)} f)^p$
\end{proof}
It follows that taking the closure of $\Omega^k({\bf B_1})$ with respect to the $L_p$-norm induces bounded maps on Banach spaces  $\Omega_p^k({\bf B_1})$ and there exists the diagram 
$$
\xymatrix{ \dots\ar[r]^-{d}&\Omega^{k-1}_p({\bf B}_1) \ar[r]^-{d}\ar[d]_{\mathcal{R}_{\varepsilon}}
&\Omega^{k}_p({\bf B}_1)\ar[r]^-{d}\ar[ld]_{\mathcal{A}_{\varepsilon}}\ar[d]_{\mathcal{R}_{\varepsilon}}
&\Omega^{k+1}_p({\bf B}_1)\ar[ld]_{\mathcal{A}_{\varepsilon}}\ar[r]^-{d}\ar[d]^{\mathcal{R}_{\varepsilon}}&\dots\\
\dots\ar[r]_-{d}&\Omega^{k-1}_p ({\bf B}_1)\ar[r]_-{d}&\Omega^{k}_p({\bf B}_1)\ar[r]_-{d}&\Omega^{k+1}_p({\bf B}_1)\ar[r]_-{d}&\dots
}
$$
with commutative squares and such morphisms that  the equation bellow holds
$$
\mathcal{R}( \omega) - \omega = d \mathcal{A(\omega)} + \mathcal{A}( d\omega).
$$
Consider bi-Lipschitz homomorphism 
$
\varphi \colon {\bf B}_1 \to B \subset {\bf R}^n
$
then we can define operators $\tilde{\mathcal R_{\varepsilon}}$ and $\tilde{\mathcal A_{\varepsilon}}$:

\begin{equation*}
\begin{split}
\xymatrix{
\Omega^k_{p}(B)\ar[r]^-{\tilde{\mathcal R_{\varepsilon}}}\ar[d]_{\varphi^*}&\bigwedge^k \Omega_{\mathscr L}(B)\\
\Omega^k_{p}({\bf B}_1)\ar[r]_-{\mathcal R_{\varepsilon}}&\Omega^k_{\rm smooth}({\bf B}_1)\ar[u]_{(\varphi^{-1})^*}
}
\end{split}
~~~~~~~~~~~~~~
\begin{split}
\xymatrix{
\Omega^k_{p}(B)\ar[r]^-{\tilde{\mathcal A_{\varepsilon}}}\ar[d]_{\varphi^*}& \Omega^{k-1}_{p}(B)\\
\Omega^k_{p}({\bf B}_1)\ar[r]_-{\mathcal A_{\varepsilon}}&\Omega^{k-1}_{p}({\bf B}_1)\ar[u]_{(\varphi^{-1})^*}
}
\end{split}
\end{equation*}
It is not hard to see that the commutative squares are squares  in the category of normed Banach spaces because all arrows are bounded maps and moreover we have
$$
\|\tilde{\mathcal{R}_{\varepsilon}}\|_p\le \tilde{C}(\varepsilon),
$$
$$
\|\tilde{\mathcal{A}_{\varepsilon}}\|_p\le \tilde{M}(\varepsilon).
$$

Just as we did above  take the closure of $\Omega^k({B})$ with respect to the $L_p$-norm it induces bounded maps on Banach spaces  $\Omega_p^k({ B})$ and there exists the diagram 
$$
\xymatrix{ \dots\ar[r]^-{d}&\Omega^{k-1}_p({ B}) \ar[rr]^-{d}\ar[d]^-{\tilde{\mathcal{R}_{\varepsilon}}}
&&\Omega^{k}_p({ B})\ar[rr]^-{d}\ar[lldd]_{\tilde{\mathcal{A}_{\varepsilon}}}\ar[d]^-{\tilde{\mathcal{R}_{\varepsilon}}}
&&\Omega^{k+1}_p({ B})\ar[lldd]_{\tilde{\mathcal{A}_{\varepsilon}}}\ar[r]^-{d}\ar[d]^-{\tilde{\mathcal{R}_{\varepsilon}}}&\dots\\
\dots &\bigwedge^{k-1} \Omega_{\mathscr L}(B)\ar@{_(->}[d]&&\bigwedge^k \Omega_{\mathscr L}(B)\ar@{_(->}[d]&&\bigwedge^{k+1} \Omega_{\mathscr L}(B)\ar@{_(->}[d]&\dots\\
\dots\ar[r]_-{d}&\Omega^{k-1}_{p}(B)\ar[rr]_-{d}&&\Omega^k_{p}(B)\ar[rr]_-{d}&& \Omega^{k+1}_{p}(B)\ar[r]_-{d}&\dots
}
$$
with commutative squares and such morphisms that  the equation bellow holds
$$
\tilde{\mathcal{R}_{\varepsilon}}( \omega) - \omega = d \tilde{\mathcal{A_{\varepsilon}}}(\omega) + \tilde{\mathcal{A_{\varepsilon}}}( d\omega).
$$

\section{Classes of differential forms on a metric simplicial complex}
Denote by ${\sf Fin}_{+}$ the category of finite nonempty sets and partial maps and by {\sf Set} the usual category of sets.
A simplicial complex $K$ can be defined as a functor
$$
{\sf K} \colon {\sf Fin}_{+}^{\sf op}\to {\sf Set}
$$
where ${\sf Fin}_{+}^{\sf op}$ is the opposite category of ${\sf Fin}_{+}$. Fixe some set $V$ and put
$$
K([n]) = \{ \rho \colon [n] \to V \mid \rho ~\text{is a partial injective function}\}
$$
In other words, elements of ${\sf K}([n]) =  K[n]$ serve as indices to $n$-simplices  and 
$$
f\colon [m] \to [n]
$$
induces the embedding of faces of $K$
$$
K(f)\colon K[n] \to K[m]
$$
as follows
$$
K(f)\langle \rho \rangle = \rho \circ f.
$$

The condition below was introduced in \cite{GKS2} in the context of studying noncompact Riemannian manifolds and its proper (in the sence of metric geometry) triangulations.  
\begin{definition}
We will call $K$ star-bounded if there exists $C>0$ such that 
for every $v \in K[0]$ the cardinality of a set
$$
\Psi_v  = \{ \iota \in {\Hom}(K[0],\,K[1]) \mid v \in \Dom \iota\}	
$$
satisfies the following 
$$
|\Psi_v|\le C.
$$

\end{definition}
Define a geometric realisation of the simplicial complex $K$ as a topological space
$$
|K| = \coprod_{i=0}^{n} (\Delta_i \times K[i])/ \sim
$$
where
$$
\Delta_n =\left\{(t_0,\dots,\,t_n)\in {\mathbb R}^{n+1}\bigg|
\sum\limits_{i = 0}^{n-1} t_i = 1,~t_i\ge 0\right\}
$$
and $\sim$ is an equivalence relation defined by gluing of simplices. 
We also can endow $|K|$ with the simplicial metric, that is Euclidian on each simplex.
From now on, we will follow the terminology of \cite{DK}.  Let every simplex of $K$ be isometric to the standard simplex in the Euclidian space.
Thus  each morphism 
$$
[k] \to [m]
$$
induces for each $m$-simplex $\sigma^m$ an isometric embedding of its face $\sigma^k$
$$
\sigma^k \to \sigma^m
$$ 
Also how was mentioned in \cite{DK}  we can introduce a length-metric on $K$ in such a manner that each simplex is isometrically embedded  in $K$. 
In  more detail, a piecewise-linear path $\gamma\colon [a,\,b] \to K$ is a path such that  its domain can be broken into finitely many intervals $[a_i,\,a_{i+1}]$ so that 
the image $\gamma([a_i,\,a_{i+1}])$ is a piecewise-linear path contained in a single simplex of $K$. The length of $\gamma$ is defined using Euclidian metric on simplices of $K$
$$
|\gamma| = \sum\limits_i |\gamma([a_i,\,a_{i+1}])|
$$
and so we can define the distance  as follows
$$
d(x,\,y) = \inf\limits_{\gamma}|\gamma| 
$$  
where the infimum is taken over all paths connecting $x$ and $y$ in the class of  piecewise-linear maps. 
\begin{remark}
The path-metric $d$ is complete and turn $K$ to a geodesic metric space.
\end{remark}
\begin{definition}
A metric simplicial complex $K$ has bounded geometry if it is connected, star-bounded and there exists $L\ge 1$ such that the length of every edge is in the interval $[L^{-1},\,L]$.
\end{definition}
Below we will assume that all complexes have bounded geometry with $L = 1$.

Let $ \mathscr L_{\rm loc}(|K|)$ be a space of locally Lipschitz functions on $|K|$. 
We require that
for every morphism 
$$
f \colon [k] \to [m]
$$
there is a restriction  ${\rm res}_{K[k],\,K[m]} \colon  \mathscr L_{\rm loc}(|K[m]|) \to \mathscr L_{\rm loc}(|K[k]|)$  
induced by  
an isometric embedding of its face
which can be implemented by
the  consecutive vanishing of  $m-k$ -sets of barycentric coordinates $t_j$ with indexes 
$j \notin \{j_0,\dots,\,j_k\}$
on every $m$-simplex $\Delta$.    

Define $ \mathscr{ C^{\infty}L}(|K|) \subset \mathscr L_{\rm loc}(|K|)$ as a space 
of  locally Lipschitz functions which are from the class $\mathscr C^{\infty}$ on every  topological space $(\Delta_n, \, \alpha)$, ~$\alpha \in K[n]$.
 
 \begin{theorem}\emph{(Rademacher's theorem)}
Let $U \subset {\mathbb R}^n$, and let ${f \colon U \to {\mathbb R}}$ be locally Lipschitz. Then $f$ is differentiable at almost every point in $U$.
\end{theorem} 

Let $K$ be a simplicial complex define barycentric coordinates 
${t_i \colon |K| \to \mathbb R}$ where
$$
\sum_{i} t_i = 1,~t_i\ge 0.
$$

The restrictions of coordinate functions to every simplex of $K$ are smooth and then germs of  function which are locally Lipschitz on $|K|$  and smooth inside simplices  
generate a correctly defined tangent space 
for 
every interior point $x$ of each simplex of $K$. If $x \in |K[n-1]|$ then every coordinate function $t_i$ 
is not differentiable at $x$ as function on $|K|$. It implies that every  Lipschitz function $f(t_0,\dots,t_n)$ is not differentiable at $x$ as well. 
In spite of that fact the Rademacher's theorem allows us to define the $n-1$-dimensional tangent space 
almost everywhere on 
the $n-1$-dimensional skeleton of our complex
using the restriction of coordinate functions to the $|K[n-1]|$. 
It follows that we can define the  tangent space 
almost everywhere on a skeleton of each dimension.  Let $\mathscr{ C^{\infty}L}_c(|K|)$ be a subspace of compactly supported functions in
$ \mathscr{ C^{\infty}L}(|K|)$. Suppose $f \in   \mathscr L_{\rm loc}(|K|)$. Then define $df$ in the sense of distributions.
Due the Rademacher's theorem  at almost every point $x$ of $|K|$ we can consider continuous germ 
$$
f -f(x)= \langle \nabla_x f,\, \sum_{i = 1}^n (\xi_i-\xi_i(x))\vec{e_i}\rangle + o\big(\big|\sum_{i = 1}^n (\xi_i-\xi_i(x))\vec{e_i}\big|\big).
$$
that implies that it is reasonable  to assign to $f$ a vector  
$$
f-f(x)  = \langle \nabla_x f,\, \sum_{i = 1}^n (\xi_i-\xi_i(x))\vec{e_i}\rangle  \mod  \mathfrak o_x.
$$
And so the function $f$  induces a cotangent vector $df$ at almost every point $x$ it implies that
$$
\int_U df \wedge h
$$
is defined for every $U \subset |K|$, where  $h$  is ${n-1}$-form  defined  on every simplex as exterior product of differential of   functions from $\mathscr{ C^{\infty}L}_c(|K|)$.
Summarizing, we can define $df$ over $U$ as a functional:
$$
df (h) =- \int_U f dh
$$
such that the following holds for every $h$
$$
\int_U df \wedge h=- \int_U f dh.
$$
As a result locally almost everywhere we have a finitely generated $L_{\infty}$-module 
and an epimorphism ${{L}^n_{\infty} \to {\Omega_{\mathscr L}}}$ induced by the  map $df \mapsto \frac{\partial f}{\partial x_1}dx_1+\dots+\frac{\partial f}{\partial x_n}dx_n$,
where ${\Omega_{\mathscr L}}$ is the $L_{\infty}$-module of  Lipschitz $1$-forms.

\begin{definition}
We will use the following notation for $L_p$-norms on  $ \mathscr L_{\rm loc}(|K|)$:
\begin{itemize}
\item[$\bullet$] 
$ f \in \mathscr L_{\rm loc}(|K|)$,~$\|f\|_{L_p} = (\sum_{T\in K([n])} \int_T |f(x)|^p dx)^{\frac{1}{p}}$
\item[] In the case $p = \infty$  we put $\|f\|_{L_{\infty}} = {\rm ess}\sup|f(x)|$
\end{itemize}
\end{definition}

\begin{definition}
We will use the following notation for $L_p$-norms on spaces of differential forms on  $|K|$:
\begin{itemize}
\item[$\bullet$] 
$ \omega \in \bigwedge^k \Omega_{\mathscr L}(K)$,~$\|\omega\|_{\Omega_{p,\,p}} = (\||\omega|\|_{L_p}^p+\||d\omega|\|_{L_p}^p)^{\frac{1}{p}} $;
\item[$\bullet$] $
\omega \in \bigwedge^k \Omega_{\mathscr L}(K)$,~ $\|\omega\|_{\Omega_{\infty,\,\infty}} = \max \{\||\omega|\|_{\Omega_\infty},\,\||d\omega|\|_{\Omega_\infty}\}$.
\end{itemize}
\end{definition}

\begin{definition}
We will define a Sobolev spaces   of differential forms on  $|K|$ as the following:
$$
\Omega^k_{p,\,p}(K) =  \bigg(\overline{\bigwedge^k \Omega_{\mathscr L}(K)}\bigg)_{\Omega_{p,\,p}} 
$$
i. e. $\Omega^k_{p,\,p}(K)$ is the closure of the graded module of Lipschitz forms with respect to the norm of  Sobolev spaces.
\end{definition}
\begin{lemma}
Let ${\Delta}$ be a simplex and $\partial \Delta$ be its boundary.   
Then any $ \omega \in \bigwedge^k \Omega_{\mathscr L}(\partial \Delta)$ can be extended to the whole $\Delta$ in such a way that
$ \tilde{\omega} \in \bigwedge^k \Omega_{\mathscr L}(\Delta)$ and $\|\tilde{\omega}\|_{\Omega^{*}_{p,\,p}(\Delta)} \le \|\omega\|_{\Omega^{*}_{p,\,p}(\partial \Delta)}$. 
\end{lemma}
\begin{proof}
Let $I = [0,\,1]$. Consider a Lipschitz  form  $ \omega \in \bigwedge^k \Omega_{\mathscr L}(\partial I^n)$. 
Our aim is to define $ \tilde{\omega} \in \bigwedge^k \Omega_{\mathscr L}(\partial I^n\times I)$ in such a manner that 
$$
\tilde{\omega}\big|_{\partial I^n} = \omega,~\text{and}~\|\tilde{\omega}\|_{\Omega^{*}_{p,\,p}(\partial I^n\times I)} \le \|\omega\|_{\Omega^{*}_{p,\,p}(\partial I^n)}.
$$

Define a functions $f\colon I \to \mathbb{R}$ as the following 
$$
t \mapsto 1-t
$$
Then we can define $\tilde{\omega}$ as the following
$$
\tilde{\omega}(x,\,t) = f(t)\omega(x).
$$

\begin{align*}
\|\tilde{\omega}\|^p_{\Omega^{*}_{p}(\partial I^n\times I)} = \int\limits_{\partial I^n\times I} |\tilde{\omega}|^p dxdt
 =  \int\limits_{I} dt\int\limits_{\partial I^n} |f(t)|^p|{\omega(x)}|^p dx \\
 = \int\limits_{0}^1 |1-t|^pdt  \int\limits_{\partial I^n} |{\omega(x)}|^p dx \le  \frac{1}{(1+p)^p}  \|\omega\|^p_{\Omega^{*}_{p}(\partial I^n)} 
 \end{align*}
$$
d\tilde{\omega} = df\wedge \omega + fd\omega = (1-t)d\omega(x) - dt\wedge\omega(x)
$$
$$
|d\tilde{\omega}| \le |1-t| |d\omega(x)| + |\omega(x)|
$$
$$
\|d\tilde{\omega}\|_{\Omega^{*}_{p}(\partial I^n\times I)} \le \|\omega\|_{\Omega^{*}_{p}(\partial I^n\times I)} + \|(1-t)d\omega\|_{\Omega^{*}_{p}(\partial I^n\times I)} 
$$
$$
\|(1-t)d{\omega}\|^p_{\Omega^{*}_{p}(\partial I^n\times I)} =
 \int\limits_{I} dt\int\limits_{\partial I^n} |1-t|^p|{d\omega(x)}|^p dx = \frac{1}{(1+p)^p} \|d\omega\|^p_{\Omega^{*}_{p}(\partial I^n)} 
$$
$$
\|d\tilde{\omega}\|_{\Omega^{*}_{p}(\partial I^n\times I)} \le \|\omega\|_{\Omega^{*}_{p}(\partial I^n)}+ \|d\omega\|_{\Omega^{*}_{p}(\partial I^n)}
$$
$$
\|d\tilde{\omega}\|^p_{\Omega^{*}_{p}(\partial I^n\times I)} \le 2^{p-1}\|\omega\|^p_{\Omega^{*}_{p}(\partial I^n)}+ 2^{p-1}\|d\omega\|^p_{\Omega^{*}_{p}(\partial I^n)}
$$
\begin{align*}
\|\tilde{\omega}\|^p_{\Omega^{*}_{p}(\partial I^n\times I)} + \|d\tilde{\omega}\|^p_{\Omega^{*}_{p}(\partial I^n\times I)} 
\le& 2^{p-1}\|\omega\|^p_{\Omega^{*}_{p}(\partial I^n)}+ 2^{p-1}\|d\omega\|^p_{\Omega^{*}_{p}(\partial I^n)} + \|\omega\|^p_{\Omega^{*}_{p}(\partial I^n)} \\
\le& (2^{p-1}+1)\|\omega\|^p_{\Omega^{*}_{p,\,p}(\partial I^n)}
\end{align*}

Assume that $\Delta$ is an $n$-dimensional simplex and $ \omega \in \bigwedge^k \Omega_{\mathscr L}(\partial \Delta)$ is a Lipschitz form.

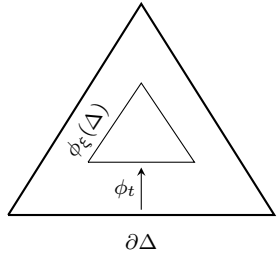
\begin{center}
\begin{tikzpicture}
\begin{scope}[scale = 0.7]
\draw[thick] (-2.5,0)--(2.5,0)--(0,4)--(-2.5,0);
\draw (-1,1)--(1,1)--(0,2.5)--(-1,1);
\draw (0,-0.5) node[font = \fontsize{8}{30}]{$\partial \Delta$};
\node[font = \fontsize{8}{30}, rotate = 57] at (-1,1.6) {$\phi_\xi(\Delta)$};
\draw[-stealth] (0,0.1) -- (0,0.9); 
\draw (-0.3, 0.5) node[font = \fontsize{8}{30}]  {$\phi_t$};
\end{scope}
\end{tikzpicture}
\end{center}

There exist  Lipschitz map $\gamma \colon \partial I^n \to  \partial \Delta $ and a couple of Lipschitz maps $g$, $h$ as illustrated below 

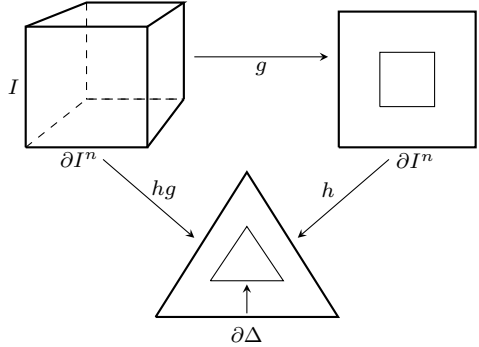
\begin{center}
\begin{tikzpicture}[scale = 0.8]
\begin{scope}[xshift = -150, scale = 0.4]
\draw[thick] (-2.5,0)--(2.5,0)--(2.5,5)--(-2.5,5)--(-2.5,0);
\draw[dashed] (0,6)--(0,2)--(4,2);
\draw[thick] (4,2)--(4,6);
\draw (-0.3,-0.5) node[font = \fontsize{8}{30}]{$\partial I^n$};
\draw (-3,2.5) node[font = \fontsize{8}{30}]{$I$};

\draw[thick] (-2.5,5)--(0,6)--(4,6)--(2.5,5);
\draw[dashed] (-2.5,0)--(0,2)--(4,2)--(2.5,0);
\draw[thick] (4,2)--(2.5,0);
\end{scope}

\draw[-stealth] (-3.5,1.5) -- (-1.3,1.5);
\draw (-2.4,1.3) node[font = \fontsize{8}{30}]{$g$};
 
\draw[-stealth] (-5,-0.2) -- (-3.5, -1.5);
\draw (-4,-0.7) node[font = \fontsize{8}{30}]{$hg$};

\begin{scope}[scale = 0.45]
\draw[thick] (-2.5,0)--(2.5,0)--(2.5,5)--(-2.5,5)--(-2.5,0);
\draw (-1,1.5)--(1,1.5)--(1,3.5)--(-1,3.5)--(-1,1.5);
\draw (0.3,-0.5) node[font = \fontsize{8}{30}]{$\partial I^n$};
\end{scope}
\draw[-stealth] (-0.3,-0.2) -- (-1.8, -1.5);
\draw (-1.3,-0.7) node[font = \fontsize{8}{30}]{$h$};

\begin{scope}[xshift = -75, yshift = -80, scale = 0.6]
\draw[thick] (-2.5,0)--(2.5,0)--(0,4)--(-2.5,0);
\draw (-1,1)--(1,1)--(0,2.5)--(-1,1);
\draw (0,-0.5) node[font = \fontsize{8}{30}]{$\partial \Delta$};
\draw[-stealth] (0,0.1) -- (0,0.9); 
\end{scope}

\end{tikzpicture}

\end{center}

We can extend $\gamma^*\omega$ as was shown above. And then put the following $\tilde{\omega} = (g^{-1}h^{-1})^* \tilde{\gamma^*\omega}$. Hence we have 
$\tilde{\omega}\big|_{\phi_{\xi}(\partial \Delta)} = 0$ and we can consider $\tilde{\omega}$ to be zero over $\phi_{\xi}(\Delta)$.
\end{proof}

\begin{lemma}
Let $K$ be an $n$-dimensional simplicial complex and $K[m]$ be its $m$-dimensional skeleton.   
Then any  $ \omega \in \Omega^{*}_{p,\,p}(K[m])$ can be extended to the whole $K$ in such a way that
$ \tilde{\omega} \in \Omega^{*}_{p,\,p}(K)$ and 
$$\|\tilde{\omega}\|_{\Omega^{*}_{p,\,p}(K)} \le \|\omega\|_{\Omega^{*}_{p,\,p}(K[m])}.$$ 
\end{lemma}
\begin{proof}
Suppose that $ \omega \in \Omega^{*}_{p,\,p}(K[m])$ and there is $\{\omega_i\} \subset  \bigwedge^k \Omega_{\mathscr L}(K[m])$ such that 
$\|\omega_i - \omega\|_{ \Omega^{*}_{p,\,p}(K[m])} \to 0$ as $i \to \infty$.  

In the light of previous lemma there exists $\tilde{\omega_i} \in \Omega^{*}_{p,\,p}(K[m+1])$ which satisfies the following estimation 
$$
\|\tilde{\omega_i}\|_{ \Omega^{*}_{p,\,p}(K[m+1])} \le \|\omega_i\|_{ \Omega^{*}_{p,\,p}(K[m])}
$$
for each $i$. It is not hard to see that $\{\tilde{\omega_i}\}$ is a Cauchy sequence since the procedure of extension is linear:
$$
\|\tilde{\omega_i} - \tilde{\omega_j}\|_{ \Omega^{*}_{p,\,p}(K[m+1])} \to 0,~i,\,j \to \infty.
$$
And so
$$
\lim \|\tilde{\omega_i}\|_{ \Omega^{*}_{p,\,p}(K[m+1])} \le \lim\|\omega_i\|_{ \Omega^{*}_{p,\,p}(K[m])} = \|\omega\|_{ \Omega^{*}_{p,\,p}(K[m])}
$$
Denote a  limit of the sequence as the following
$$
\lim\tilde{\omega_i} = \tilde{\omega}
$$
Repeating this construction for every dimension as a result we obtain an extension to the  whole complex.
\end{proof}
\begin{remark}
It is not hard to see that the same argument holds  for $S\mathscr{L}_p(K)$.
\end{remark}
\begin{lemma}
Let $S^k$ be a $k$-sphere and $B \subset S^k$ be a $k$-ball.   Any Lipschitz  $k$-form $\omega \in \Omega^k(B)$ can be extended by zero to the Lipschitz $k$-form on $S^k$. 
\end{lemma}
\begin{proof}
There is a homotopy $\varphi \colon S^k \times [0,\,1] \to S^k$ such that  $\varphi(x,\,0) = {\rm Id}$ and
$\varphi(B,\,1) = B'$ where $B\cap B' = \varnothing$. 

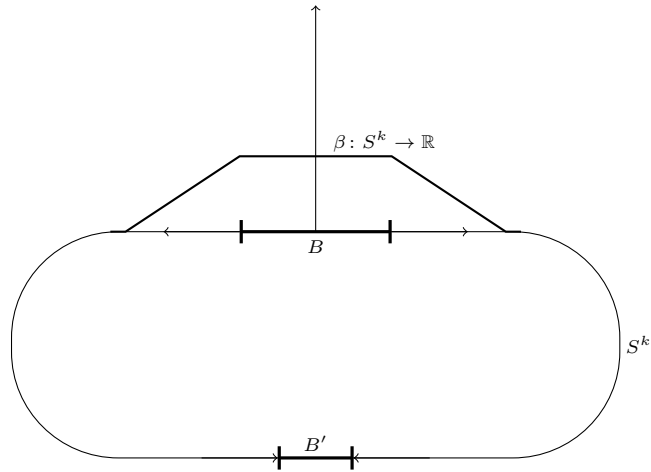
\begin{center}
\begin{tikzpicture}
\draw[rounded corners = 40pt] (-4, 0) rectangle (4,-3);  
\draw[very thick,|-|] (-1,0)--(1,0);
\draw[very thick,|-|] (-0.5,-3)--(0.5,-3);
\draw[->](0,0) -- (0,3);
\draw[thick] (-2.7,0)--(-2.5,0)--(-1,1)--(1,1)--(2.5,0)--(2.7,0);
\draw (0.9,1.2) node[font = \fontsize{7}{30}]{$\beta\colon S^k \to \mathbb R$};
\draw (4.25,-1.5) node[font = \fontsize{7}{30}]{$S^k$ };
\draw (0,-0.2) node[font = \fontsize{7}{30}]{$B$ };
\draw (0,-2.8) node[font = \fontsize{7}{30}]{$B'$ };
\draw[->](-1,0) -- (-2,0);
\draw[->](-1.5,-3) -- (-0.5,-3);
\draw[->](1,0) -- (2,0);
\draw[->](1.5,-3) -- (0.5,-3);
\end{tikzpicture}
\end{center}

\vspace{5mm}

So we can extend any $k$-form $\omega \in \Omega^k(B)$ by a zero $k$-form to the whole $S^k$. Define first $\omega$ over $B'$  as follows 
 $$
\tilde{\omega} = (\varphi_{t=1}^{-1})^* \omega.
 $$
 Let $\omega = f^0 df^1\wedge \dots \wedge df^k$. Then for every $a \in \partial B$ we can  put  $f^i(\varphi(a,t)) = f^i(a)$.
So as a result we have the following. Let $v_1$ be a tangent vector to the curve $\varphi(a,\,t)\colon [0,\,1] \to S^k$. Consider a basis $v_1,\dots, \,v_k$. Then 
$\nabla f^i$ has the zero component corresponding to the direction $v_1$. So we obtain 
$$
df^i = f^i_{v_2}v^*_2+\dots+f^i_{v_k}v^*_k.
$$
And as a result
$$
df^1\wedge \dots \wedge df^k = 0.
$$
Now let $\beta \colon S^k \to \mathbb R$ be a `bump' function such that $\beta(B) = \{1\}$  and ${\rm supp}(\beta) \subset S^k \setminus B'$. 
Hence 
$$
{\rm supp}(\beta \tilde{\omega}) \subseteq B.
$$
$$
d(\beta \tilde{\omega}) = 0
$$
\end{proof}
\section{de Rham operators on simplicial complexes}
\begin{theorem}
Let $K$ be a complex of bounded geometry with $L = 1$. Then there exists the diagram
$$
\xymatrix{ \dots\ar[r]^-{d}&\Omega^{k-1}_{p,\,p}({ K}) \ar[rr]^-{d}\ar[d]^-{\mathscr{R}}
&&\Omega^{k}_{p,\,p}({ K})\ar[rr]^-{d}\ar[lldd]_{\mathscr{A}}\ar[d]^-{\mathscr{R}}
&&\Omega^{k+1}_{p,\,p}({ K})\ar[lldd]_{\mathscr{A}}\ar[r]^-{d}\ar[d]^-{\mathscr{R}}&\dots\\
\dots &S\mathscr{L}^{k-1}_p (K)\ar@{_(->}[d]&&S\mathscr{L}^{k}_p(K)\ar@{_(->}[d]&&S\mathscr{L}^{k+1}_p(K)\ar@{_(->}[d]&\dots\\
\dots\ar[r]_-{d}&\Omega^{k-1}_{p,\,p}(K)\ar[rr]_-{d}&&\Omega^k_{p,\,p}(K)\ar[rr]_-{d}&& \Omega^{k+1}_{p,\,p}(K)\ar[r]_-{d}&\dots
}
$$
with commutative squares in the category of Banach spaces ${\sf Ban}_{\infty}$. Moreover the following holds
$$
\mathscr R  - \mathrm{Id}_{\Omega^*_{p,\,p}} = d\mathscr A + \mathscr A d
$$
\end{theorem}
In order to prove the above theorem we will state a number of lemmas about the arrows of the diagram.  
Let $K$ be a star-bounded complex. Assume that $K'$ is the first barycentric subdivision of $K$. Let  $\Sigma'_i$ be the star of vertex $e_i$ in $K'$.  
Let $\varphi_i$ be a bi-Lipschitz homeomorphism $\varphi_i\colon {\rm Int}\,\Sigma_i \to U$ such that ${\bf B_1}\subset U$ and   $\Sigma^{\prime}_i \subset  {\rm Int}\,\varphi^{-1}({\bf B}_1) $.

Given $\varepsilon > 0$, define operators $\mathcal R_i$ and $\mathcal A_i$
$$
\mathcal R_i \omega = \begin{cases}
 (\varphi_i^{-1})^* {\mathcal R_{\varepsilon}} \varphi_i^* \omega~\text{on}~\Sigma_i\\
\omega,\text{~otherwise}
\end{cases};~\mathcal A_i \omega = \begin{cases}
(\varphi_i^{-1})^* {\mathcal A_{\varepsilon}} \varphi_i^*\omega ~\text{on}~\Sigma_i\\
0, \text{~otherwise}
\end{cases}
$$ 

Consider operators
$$
\mathscr{R}\omega  = \lim\limits_{i\to\infty} \mathcal R_1\mathcal R_2\dots \mathcal  R_i \omega
$$
$$
\mathscr{A}\omega  = \sum \limits^{\infty}_{i=1} \mathcal R_1\mathcal R_2\dots \mathcal  R_{i-1} \mathcal A_i\omega
$$

\begin{lemma}
The arrow $\Omega^{*}_{p,\,p}({ K}) \xrightarrow{\mathscr{R}} S\mathscr{L}^{k}_p(K)$  is a morphism in the category ${\sf Ban}_{\infty}$,
namely, $\mathscr R$ is  a bounded operator. 
\end{lemma}
\begin{proof}
Consider a star $\Sigma_i$ of $K$.   Assume that 
there is a set $X_i$ such that  \linebreak
${ \Sigma^{\prime}_i \subset {\rm Int }\,X_i \subset {\rm Int}\, \varphi^{-1}({\bf B}_1) } $.
We can represent  $\omega$ as a sum $\omega = \omega_1 + \omega_2$ 

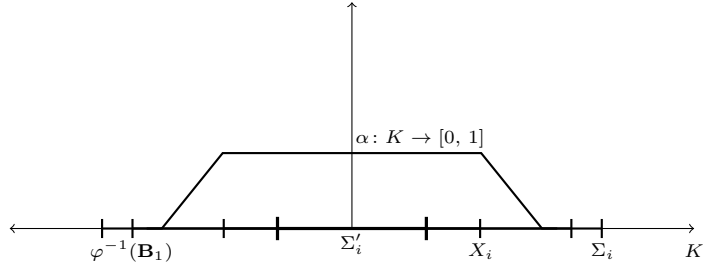
\begin{center}
\begin{tikzpicture}
\draw[very thick,|-|] (-1,0)--(1,0);
\draw[thick,|-|] (-1.7,0)--(1.7,0);
\draw[thick,|-|] (-3.3,0)--(3.3,0);
\draw[thick,|-|] (-2.9,0)--(2.9,0);
\draw[->](0,0) -- (0,3);
\draw[thick] (-2.7,0)--(-2.5,0)--(-1.7,1)--(1.7,1)--(2.5,0)--(2.7,0);
\draw (0.9,1.2) node[font = \fontsize{7}{30}]{$\alpha\colon K \to [0,\,1]$};
\draw (0,-0.2) node[font = \fontsize{7}{30}]{$\Sigma^{\prime}_i$ };
\draw (1.7,-0.3) node[font = \fontsize{7}{30}]{$X_i$ };
\draw (-2.9,-0.3) node[font = \fontsize{7}{30}]{$ \varphi^{-1}({\bf B}_1)$ };
\draw (3.3,-0.3) node[font = \fontsize{7}{30}]{$\Sigma_i$ };
\draw (4.5,-0.3) node[font = \fontsize{7}{30}]{$K$ };
\draw[->](-1,0) -- (-4.5,0);
\draw[->](1,0) -- (4.5,0);
\end{tikzpicture}
\end{center}
where $\omega_1 = \alpha\omega$ and $\omega_2 = (1-\alpha)\omega$, i. e. ${{\rm supp} (\omega_2) \subset K \setminus X_i}$.

For any $\mathcal R_j$ and $\eta \in \Omega^k(K)$ such that  
${{\rm supp} (\eta) \subset K \setminus X_i}$, choosing $\varepsilon>0$ sufficiently small, 
we can achieve  ${{\rm supp} (\mathcal R_j \eta) \subset K \setminus \Sigma^{\prime}_i}$  
that implies
$R_j \eta  = 0$ on $\Sigma^{\prime}_i$. Due to this fact,  for each $j$ 
we can choose $\varepsilon_j$ in the definition of  the operator $\mathcal R_j$ in such a way that 
${{\rm supp} ( \mathcal R_{1}\dots  \mathcal R_{j} \omega_2) \subset K \setminus \Sigma^{\prime}_i}$  and correspondingly
${ \mathscr R \omega  = \mathscr  R\omega_1}$ on $\Sigma^{\prime}_i$. 

Let  $\Sigma_i$ be spanned by points $\{e_{j_1}, \dots,\, e_{j_n}\}$. 
For every form $\theta$ such that ${\rm{supp}(\theta) \subset \varphi^{-1}({\bf B}_1) \subset \Sigma_i}$ 
we can see that only for $k \in \{{j_1}, \dots,\, {j_n}\}$ the operator $\mathcal R_k$ is distinct from the identity.
Choosing sufficiently small $\varepsilon$ each time we face such an operator $R_k \in \{R_{j_1}\,\dots,\,R_{j_n}\}$ in the composition ${\mathcal R_{1} \mathcal  R_{2} \dots \mathcal R_j }$
we can obtain a map  preserving the support of a form derived at  this step  inside ${\rm Int}\, \Sigma_i$.
Then we have ${\mathcal R_{1}\dots\mathcal  R_{j} \omega_1 = \mathcal R_{j_1}\dots\mathcal  R_{j_n} \omega_1}$ for each $j$. 
 And so $\mathscr R \omega_1 = \mathcal R_{j_1}\dots\mathcal  R_{j_n} \omega_1$.

We know that 
$$
\mathcal  R_{j_{k}}  \colon \Omega^{*}_p(\Sigma_i) \to  \Omega^{*}_p(\Sigma_i) 
$$
and moreover  the operator 
$$
\|\mathcal  R_{j_{k},\,(\varepsilon)}\|_{p} \le 1 + \varepsilon,~ \varepsilon \to 0.
$$
Then  there exists $\varepsilon > 0$ such that 
$$
 \| \mathscr R \|_p = \|\mathcal R_{j_1}\dots\mathcal  R_{j_n}\|_p \le 1 + O(\varepsilon),~ \varepsilon \to 0.
$$
We should make a note that $d \mathscr R = \mathscr R d$ and the above argument holds for $d \omega$.

As a result for each $i$ we have
$$
 \| {\rm res}_{\Sigma_i^{\prime},\,K} \circ \mathscr R \omega \|_{\Omega^k_{p,\,p}(\Sigma_i^{\prime})}  
 \le (1 + O(\varepsilon_i))\|{\rm res}_{\Sigma_i,\,K} \omega \|_{\Omega^k_{p,\,p}(\Sigma_i)},
 ~ \varepsilon_i \to 0.
$$
Due to the star-boundedness of the complex we can choose $\varepsilon = \min_i {\varepsilon_i}$. Let $\omega \in \Omega^k_{p,\,p}(K)$
\begin{align*}
\| \mathscr R \omega\|_{\Omega^k_{p,\,p}(K)} = \sum_{i}  \| {\rm res}_{\Sigma^{\prime}_i,\,K} \circ \mathscr R \omega \|_{\Omega^k_{p,\,p}(K)}
& \le \sum_{i}(1 + O(\varepsilon))\|{\rm res}_{\Sigma_i,\,K} \omega \|_{\Omega^k_{p,\,p}(\Sigma_i)}\\
\le \frac{(1 + O(\varepsilon))}{n}&\|\omega \|_{\Omega^k_{p,\,p}(K)}
\end{align*}
In the light of what we have just said, $\mathscr R$ is a bounded map:
$$
\mathscr R \colon \Omega^{*}_{p,\,p}(K) \to \Omega^{*}_{p,\,p}(K).
$$
Moreover, it is not hard to see that  $\mathscr R \colon \Omega^{*}_{p,\,p}(K) \to S\mathscr{L}_p^{*}(K)$. Indeed, 
we know that every $\mathcal R_i =  (\varphi_i^{-1})^* {\mathcal R_{\varepsilon}} \varphi_i^* $, then 
$$
 {\mathcal R_{\varepsilon}} \varphi_i^* \colon  \Omega^{*}_{p,\,p}(\Sigma_i) \to \Omega^{*}_{\scriptscriptstyle{\rm smooth}}(U), 
$$
and $(\varphi_i^{-1})^*$ is a Lipschitz piecewise smooth map. 
\end{proof}

\begin{lemma}
For every $m$-dimensional skeleton $K[m]$ of $K$ the operator \linebreak
${{\rm res}_{K[m],\, K}\circ \mathscr R}$
is a morphism $\Omega^{*}_{p,\,p}({ K}) \to S\mathscr{L}^{k}_p(K[m])$ in $\mathsf{Ban}_{\infty}$ (a bounded operator).
\end{lemma}
\begin{proof}
From now on we will follow the notation stated in the proof of  \emph{Lemma 2}. 
In particular, let a star  $\Sigma_i$ be spanned by points $\{e_{j_1}, \dots,\, e_{j_n}\}$. Every 
$n$-dimensional simplex $\sigma$ can be covered by stars of $K^{\prime}$ 
$$\sigma \subset \bigcup_{k = 0}^{n-1} \Sigma^{\prime}_{j_k}$$ 
that applying \emph{Lemma 3} and the argument from the proof of \emph{Lemma 2}  implies
 $$\| {\rm res}_{\sigma,\, K} \circ \mathscr R \omega \|_{\Omega^{*}_{\infty}(\sigma)} 
 \le \sum \limits_{k = 0}^{n} \| {\rm res}_{\Sigma^{\prime}_{j_k},\, K} \circ \mathscr R \omega \|_{\Omega^{*}_{\infty}(\Sigma^{\prime}_{j_k})}
  \le C \sum \limits_{k = 0}^{n} \| {\rm res}_{\Sigma_{j_k},\, K}  \omega \|_{\Omega^{*}_{p}(\Sigma_{j_k})}$$
and 
\begin{align*}
\sum_i \| {\rm res}_{\sigma_i,\, K} \circ \mathscr R \omega \|_{\Omega^{*}_{\infty}(\sigma_i)} 
  \le &C \sum_i \sum \limits_{k = 0}^{n} \| {\rm res}_{\Sigma^i_{j_k},\, K}  \omega \|_{\Omega^{*}_{p}(\Sigma^i_{j_k})} \\
  \le C  N \sum_j \| {\rm res}_{\Sigma_{j},\, K}  \omega \|_{\Omega^{*}_{p}(\Sigma_{j})}
    \le &C  N n \sum_i \| {\rm res}_{\sigma_{i},\, K}  \omega \|_{\Omega^{*}_{p}(\sigma_{i})} = C^{\prime}\|\omega \|_{\Omega^{*}_{p}(K)}.
  \end{align*}
 Hence 
 \begin{align*}
 \|{\rm res}_{K[m],\, K}\circ \mathscr R \omega\|_ {\Omega^{k}_p(K[m])} = \sum_{i} \|{\rm res}_{\tau_i,\, K}\circ \mathscr R \omega\|_ {\Omega^{k}_p(\tau_i)}\\
 \le\sum_{i}  ({\rm mes}\, \tau_i)^{\frac{1}{p}} \|{\rm res}_{\tau_i,\, K}\circ \mathscr R \omega\|_ {\Omega^{k}_{\infty}(\tau_i)}.
 \end{align*}
 It is not hard to see that
 $$
 \|{\rm res}_{\tau_i,\, K}\circ \mathscr R \omega\|_ {\Omega^{k}_{\infty}(\tau_i)} 
 \le \sum_{\sigma \in K[n],\,\tau_i \hookrightarrow \sigma}\|{\rm res}_{\sigma,\, K}\circ \mathscr R \omega\|_ {\Omega^{k}_{\infty}(\sigma)} 
$$
As a result we have
\begin{align*}
\|{\rm res}_{K[m],\, K}\circ \mathscr R \omega\|_ {\Omega^{k}_p(K[m])} 
&\le \bigg( \frac{\sqrt{m+1}}{m!\sqrt{2^{m}}}\bigg)^{\frac{1}{p}} 
\sum_{i}  \sum_{\sigma \in K[n],\,\tau_i \hookrightarrow \sigma}\|{\rm res}_{\sigma,\, K}\circ \mathscr R \omega\|_ {\Omega^{k}_{\infty}(\sigma)} \\
&\le \bigg( \frac{\sqrt{m+1}}{m!\sqrt{2^{m}}}\bigg)^{\frac{1}{p}} \binom{n+1}{m+1}C^{\prime}\|\omega \|^p_{\Omega^{*}_{p}(K)}.
\end{align*}
Here we should again  make a note that $d \mathscr R = \mathscr R d$ and the above argument holds for $d \omega$.
\end{proof}
\begin{lemma}
The arrow 
$\Omega^{k}_{p,\,p}({ K}) \xrightarrow{\mathscr{A}} \Omega^{k-1}_{p,\,p}({ K})$ is a morphism in  the category $\mathsf{Ban}_{\infty}$. 
\end{lemma}
\begin{proof}
Let $\omega \in \Omega^k_{p,\,p}(M)$ consider  $\mathcal A_i\omega$. It is not hard to see that  \linebreak
${\rm{supp}(\mathcal A_i \omega) \subset \varphi^{-1}({\bf B}_1)}$. 
Indeed, $\mathfrak s_{tv}^*$  acts on  the complement $K \setminus \varphi^{-1}({\bf B}_1)$  leaving points fixed. 
It follows that  $ \frak X_v = 0$ and,  consequently, $\iota_{\frak X_v}$ is a zero map at every point belonging $K \setminus \varphi^{-1}({\bf B}_1)$.

Let  $\Sigma_i$ be spanned by points $\{e_{j_1}, \dots,\, e_{j_n}\}$. 
For every form 
compactly supported inside  ${\varphi^{-1}({\bf B}_1) \subset \Sigma_i}$ there are only a fixed number of operators, namely  $\mathcal R_{j_1},\dots,\, R_{j_n}$, which are distinct from the identity.
Similarly to the above we can choose sufficiently small $\varepsilon$  for each operator $R_{j_k}$ in the composition ${\mathcal R_{1}\dots\mathcal  R_{i-1}\mathcal A_i }$,
that allows us to preserve the support of a form derived under every partial composition inside ${\rm Int}\, \Sigma_i$.
Then we have \linebreak
${\mathcal R_{1}\dots\mathcal  R_{i-1}\mathcal A_i \omega = \mathcal R_{j_1}\dots\mathcal  R_{j_n} \mathcal A_i \omega}$ and 
${\mathcal R_{j_1}\dots R_{j_n} \mathcal A_i \omega = 0}$  outside $\Sigma_i$. Now we can estimate the norm $\|\mathscr A \omega\|_{\Omega^{k-1}_p(K)}$:
$$
\|\mathcal R_{j_1}\dots R_{j_n} \mathcal A_i \omega\|_{\Omega^{k-1}_p(\Sigma_i)} \le C^{n}_p M_p\|\omega\|_{\Omega^{k}_p(\Sigma_i)}.
$$
\begin{align*}
&\|\mathscr A \omega\|^p_{\Omega^{k-1}_{p,\,p}(K)} = \|\mathscr A \omega\|^p_{\Omega^{k-1}_{p,\,p}(K)} + \| d\mathscr A \omega\|^p_{\Omega^{k-1}_{p,\,p}(K)} \\
= &\int\limits_{K}\left|\sum_{i} \mathcal R_{1}\dots R_{i-1} \mathcal A_i \omega  \right|^p dx + \int_{K}\left|\sum_{i} d\mathcal R_{1}\dots R_{i-1} \mathcal A_i \omega  \right|^p dx\\
= &\sum_{\sigma \in K[n]} \bigg( \int\limits_{\sigma} \left|\sum_{i} {\rm res}_{\sigma,\, K} \mathcal R_{1}\dots R_{i-1} \mathcal A_i \omega  \right|^p dx\\
+ &\int\limits_{\sigma}\left|\sum_{i} {\rm res}_{\sigma,\, K} d\mathcal R_{1}\dots R_{i-1} \mathcal A_i \omega  \right|^p dx\bigg)
\end{align*}
Every simplex $\sigma = \{e_0,\dots,\,e_n\}$ is the intersection of stars assigned to its vertices $\sigma = \bigcap_{e_i\in  \{e_0,\dots,\,e_n\}} \Sigma_i$, and moreover 
$\sigma$ lies in no other star. And it follows that 
$$
\sum_{i} {\rm res}_{\sigma,\, K} \mathcal R_{1}\dots R_{i-1} \mathcal A_i \omega  = \sum_{e_j \in \{e_0,\dots,\,e_n\}} {\rm res}_{\sigma,\, K} \mathcal R_{1}\dots R_{j-1} \mathcal A_j \omega 
$$
and
$$
\sum_{i} {\rm res}_{\sigma,\, K} d \mathcal R_{1}\dots R_{i-1} \mathcal A_i \omega  = \sum_{e_j \in \{e_0,\dots,\,e_n\}} {\rm res}_{\sigma,\, K} d \mathcal R_{1}\dots R_{j-1} \mathcal A_j \omega 
$$
Hence
\begin{align*}
\|\mathscr A \omega\|^p_{\Omega^{k-1}_{p,\,p}(K)} 
=& \sum_{\sigma \in K[n]}  \int\limits_{\sigma} \left| \sum_{e_j \in \{e_0,\dots,\,e_n\}}{\rm res}_{\sigma,\, K} \mathcal R_{1}\dots R_{j-1} \mathcal A_j \omega  \right|^p dx\\ 
+&  \sum_{\sigma \in K[n]} \int\limits_{\sigma}\left| \sum_{e_j \in \{e_0,\dots,\,e_n\}} {\rm res}_{\sigma,\, K} d\mathcal R_{1}\dots R_{j-1} \mathcal A_j \omega  \right|^p dx\\
= (n+1)^{p-1} &\sum_{\sigma \in K[n]}  \int\limits_{\sigma}  \sum_{e_j \in \{e_0,\dots,\,e_n\}}|{\rm res}_{\sigma,\, K} \mathcal R_{1}\dots R_{j-1} \mathcal A_j \omega  |^p dx\\
+ (n+1)^{p-1} &\sum_{\sigma \in K[n]}  \int\limits_{\sigma} \sum_{e_j \in \{e_0,\dots,\,e_n\}} |{\rm res}_{\sigma,\, K} d\mathcal R_{1}\dots R_{j-1} \mathcal A_j \omega  |^p dx
\end{align*}
Let us check the following
\begin{align*}
  \int\limits_{\sigma}  &\sum\limits_{e_j \in \{e_0,\dots,\,e_n\}}|{\rm res}_{\sigma,\, K} \mathcal R_{1}\dots R_{j-1} \mathcal A_j \omega  |^p dx\\
  = &\sum_{e_j \in \{e_0,\dots,\,e_n\}}   \int\limits_{\sigma}|{\rm res}_{\sigma,\, K} \mathcal R_{1}\dots R_{j-1} \mathcal A_j \omega  |^p dx\\
  \le &\sum_{e_j \in \{e_0,\dots,\,e_n\}}   \int\limits_{\Sigma_j}| \mathcal R_{1}\dots R_{j-1} \mathcal A_j \omega  |^p dx\\
  = &\sum_{e_j \in \{e_0,\dots,\,e_n\}}   \int\limits_{\Sigma_j}| \mathcal R_{j_0}\dots R_{j_n} \mathcal A_j \omega  |^p dx\\
  = &\sum_{e_j \in \{e_0,\dots,\,e_n\}}  \|\mathcal R_{j_0}\dots R_{j_n} \mathcal A_j \omega  \|^p_{\Omega^{k-1}_p(\Sigma_j)}\\
  \le &\sum_{e_j \in \{e_0,\dots,\,e_n\}}  C^{n}_p M_p\|\omega\|_{\Omega^{k}_p(\Sigma_j)}
\end{align*}
That implies
\begin{align*}
 \sum_{\sigma \in K[n]}  \int\limits_{\sigma}  &\sum_{e_j \in \{e_0,\dots,\,e_n\}}|{\rm res}_{\sigma,\, K} \mathcal R_{1}\dots R_{j-1} \mathcal A_j \omega  |^p dx\\
\le  \sum_{\sigma \in K[n]}  &\sum_{e_j \in \{e_0,\dots,\,e_n\}}  C^{\prime}\|\omega\|_{\Omega^{k}_p(\Sigma_j)}
\le C^{\prime} (n+1)\|\omega\|_{\Omega^{k}_p(K)}
\end{align*}
 
Similarly, we can estimate $\| d\mathscr A \omega\|^p_{\Omega^{k-1}_{p,\,p}(K)}$.
And as  a result we have
$$
\|\mathscr A \omega\|^p_{\Omega^{k-1}_{p,\,p}(K)} \le C^{\prime\prime}\| \omega\|^p_{\Omega^{k}_{p,\,p}(K)}
$$
\end{proof}
Summing up the results of preceding lemmas we can conclude that  the {\it Theorem} 4 holds.
\begin{theorem} There is an isomorphism
$H^{k}(S\mathscr{L}^{*}_p(K))\cong  H^{k} (\Omega^{*}_{p,\,p}({ K})) $ in the category ${\sf Vec}_{\mathbb R}$.
\end{theorem}
\begin{proof}
Let $\omega$ be a form lying in $\Omega^{*}_{p,\,p}({ K})$ such that $\omega \in \ker d$. 
Due to the existence of homotopy between $\mathscr R$ and $\rm Id_{\Omega^{*}_{p,\,p}({ K})}$ 
we can see that $\mathscr R \omega \in S\mathscr{L}^{*}_p(K)$ and $\omega$ belong to the same cohomologycal class in $H^{k} (\Omega^{*}_{p,\,p}({ K}))$ . 
\end{proof}
\section{de Rham theorem}
Let $K$ be an $n$-dimension simplicial complex and $S\mathscr{L}^{*}_p(K)$ be a chain complex of differential forms of  Lipschitz-Sullivan's type on $K$. Put the following
$$
\mathscr I\colon \omega \mapsto f_{\omega}(\sigma) = \int\limits_{\sigma} \omega
$$
In order to prove the lemma below we will introduce an $\mathbb R$-linear map
$$
\mathscr W \colon C^k_p(K) \to S\mathscr{L}^k_p(K). 
$$
Due to Hassler Whitney's works, where such transformation was introduced, see for example \cite{HW}, an object ${\rm im}\, \mathscr W$ is usually called  the set of Whitney's forms.
\begin{lemma}
Let $K$ be a simplicial complex there exists a short strictly exact sequence of the following type in the category of Banach spaces ${\sf Ban}_{\infty}$:
$$
\xymatrix{ 
0\ar[r]&{\rm ker}^k (\mathscr I) \ar[r] &S\mathscr{L}^k_p(K)\ar[r]^{\mathscr I}&C^k_p(K)\ar[r]&0}
$$
In other words, the map $\mathscr I$ is a split epimorphism
$$S\mathscr{L}^k_p(K) \cong C^k_p(K) \oplus {\rm ker}^k( \mathscr I)$$ 
\end{lemma}
\begin{proof}
First of all, we should confirm that $\mathscr I$ is in fact a morphism between $L_p$-spaces. Indeed, 
\begin{align*}
\|\mathscr I\omega\|_{C^k_p(K)}^p = &\sum\limits_{\tau \in K[k]} \left| \int\limits_{\tau}\omega\right|^p \le
 \sum\limits_{\tau \in K[k]}  \left(\int\limits_{\tau}|\omega|\right)^p \\
\le\{\rm{mes}\,\tau\}^{p-1}  \bigg( &\sum\limits_{\tau\in K[k]} \int\limits_{\tau}|\omega|^p\bigg)
\le \left( \frac{\sqrt{k+1}}{k!\sqrt{2^{k}}} \right)^{p-1} \|\omega\|^p_{S\mathscr{L}^k_p(K)}
\end{align*}

Assume $c \in C^k_p(K)$, let us represent $c$ as a sum in the Banach space which converges absolutely  
$$
c = \sum_{\sigma \in K([k])} c(\sigma)\chi_{\sigma},~\text{where}~\chi_{\sigma}(\sigma') = \begin{cases} 1,~\text{if}~\sigma' = \sigma;\\0,~\text{otherwise},\end{cases}
$$
that is, $\sum_{\sigma \in K([k])} |c(\sigma)|^p < \infty$.  
In other words we can define $C^k_p(K)$ as a closure of $\mathbb R$-vector space $C_c^k(K)$ generated by elements $\{ \chi_\sigma\}_{\sigma \in K([k])}$, 
or equivalently   a closure of $\mathbb R$-vector space of compactly supported cochains. 

Define ${\mathscr W}$ as the following. Firstly, let us write it down for elements of basis
$$
{\mathscr W}(\chi_{\sigma}) = k!\sum_{i=0}^k (-1)^i t_i dt_0\wedge\dots\wedge dt_{i-1}\wedge dt_{i+1}\wedge \dots \wedge dt_{k},
$$
where $\{t_i\}$ is a set of barycentric coordinates assigned to vertices of the simplex $\sigma$.    
The ${\rm supp}({\mathscr W}(\chi_{\sigma}))$ contains only one simplex  belonging $k$-skeleton of K, namely $\sigma$. 
It implies that any sum ${\mathscr W}(\chi_{\sigma})+{\mathscr W}(\chi_{\sigma'})$ can be uniquely restricted to every $k$-simplex.
Then we can extend our map to the space $C_c^k(K)$ by linearity. It is clear that  ${\mathscr W} (C_c^k(K))$ is a set of  compactly supported differential forms, and moreover, the support of 
every form ${\mathscr W}(\chi_{\sigma})$
contains at most $N$ simplexes due to star-boundedness of the complex.  
Let us check that ${\mathscr W}\colon C^k_c(K) \to S\mathscr{L}^k_p(K)$ is a bounded map.  Assume that 
$$
c = \sum_{i = 1}^{m} c(\tau_i) \chi_{\tau_i}
$$
then 
$$
\|{\mathscr W}(c)\|^p_{\Omega_{\infty,\,\infty}(\sigma)} \le \binom{n+1}{k+1} ^{p-1}\sum_{i = 1}^{m} |c(\tau_i)|^p \|{\mathscr W}(\chi_{\tau_i})\|^p_{\Omega_{\infty,\,\infty}(\sigma)}.
$$
As mentioned above, there is only a finite number $q$ of simplexes belonging the support $\bigcup_{i  = 1}^m {\rm supp} ({\mathscr W}(\chi_{\tau_i}))$ of  ${\mathscr W}(c)$.
\begin{align*}
&\|{\mathscr W}(c)\|^p_{SL_{p}(K)} = \sum_{j = 1}^{q}\|{\mathscr W}(c)\|^p_{\Omega_{\infty,\,\infty}(\sigma_j)}\\
&\le  \binom{n+1}{k+1} ^{p-1}\sum_{j = 1}^{q} \sum_{i = 1}^{m} |c(\tau_i)|^p \|{\mathscr W}(\chi_{\tau_i})\|^p_{\Omega_{\infty,\,\infty}(\sigma_j)}
\end{align*}
It is not hard to see that there is a constant $C$ such that
$$
\|{\mathscr W}(\chi_{\tau_i})\|^p_{\Omega_{\infty,\, \infty}(\sigma_j)} \le C,~\text{if $\tau_i$ is a $k$-face of $\sigma_j$},
$$
and
$$
\|{\mathscr W}(\chi_{\tau_i})\|^p_{\Omega_{\infty,\,\infty}(\sigma_j)}  = 0,~\text{otherwise}.
$$
As a result we have
\begin{align*}
\sum_{j = 1}^{q} \sum_{i = 1}^{m} |c(\tau_i)|^p \|{\mathscr W}(\chi_{\tau_i})\|^p_{\Omega_{\infty,\,\infty}(\sigma_j)}
&=\sum_{i = 1}^{m} |c(\tau_i)|^p  \sum_{j = 1}^{q} \|{\mathscr W}(\chi_{\tau_i})\|^p_{\Omega_{\infty,\,\infty}(\sigma_j)}\\
\le &NC \sum_{i = 1}^{m} |c(\tau_i)|^p.
\end{align*}
Hence the map 
$$
{\mathscr W}\colon C^k_c(K) \to S\mathscr{L}^k_p(K)
$$
enjoys the following property 
$$
\|{\mathscr W}(c)\|^p_{S\mathscr{L}_p(K)} \le NC  \binom{n+1}{k+1} ^{p-1} \|c\|_{C^k_p(K)}
$$
Now we can extend ${\mathscr W}$ to the closure $\overline{(C^k_c(K))}_{\|\|_{C^k_p}} = C^k_p(K)$ by 

$$
{\mathscr W}(c) = \sum_{\sigma \in K([k])} c(\sigma){\mathscr W}(\chi_{\sigma}) = \lim_{i\to \infty} {\mathscr W}(c_i),~\text{as}~c_i \xrightarrow[i\to \infty]{}c,~\{c_i\} \subset C^k_c(K),
$$
and  it follows that
$$
{\mathscr W} \colon C^k_p(K) \to S\mathscr{L}^k_p(K).
$$
It remains to verify that  $\mathscr I$ is a retraction. First, let $\sigma \in K[k]$  it is clear that 
${{\rm supp} ({\mathscr W }(\chi_\sigma))\cap K[k] = \{\sigma\}}$ and so
$$
 \langle \mathscr I \circ {\mathscr W}(\chi_\sigma) \rangle (\sigma')= 0,~\text{if}~\sigma' \ne \sigma. 
$$
Let 
$$
{\mathscr W}(\chi_{\sigma}) = k!\sum_{i=0}^k (-1)^i t_i dt_0\wedge\dots\wedge dt_{i-1}\wedge dt_{i+1}\wedge \dots \wedge dt_{k},
$$
so we have $t_0 = 1- \sum_{i=1}^k  t_i$ inside $\sigma$. As a result we can write 
\begin{align*}
\frac{1}{k!} &\int_{\sigma} {\mathscr W}(\chi_{\sigma}) = \int_{\sigma}(1- \sum_{i=1}^k  t_i) dt_1\wedge\dots \wedge dt_{k}  \\
+&\int_{\sigma}\sum_{i=1}^k (-1)^i t_i d (1- \sum_{j=1}^k  t_j) \wedge\dots\wedge dt_{i-1}\wedge dt_{i+1}\wedge \dots \wedge dt_{k} \\
=  &\int_{\sigma}(1- \sum_{i=1}^k  t_i) dt_1\wedge\dots \wedge dt_{k}\\
+ &\int_{\sigma}\sum_{i=1}^k (-1)^{i+1} t_i d (\sum_{j=1}^k  t_j) \wedge\dots\wedge dt_{i-1}\wedge dt_{i+1}\wedge \dots \wedge dt_{k} \\
=  &\int_{\sigma}(1- \sum_{i=1}^k  t_i) dt_1\wedge\dots \wedge dt_{k}\\
+&\int_{\sigma}\sum_{i=1}^k (-1)^{i+1} t_i dt_i \wedge\dots\wedge dt_{i-1}\wedge dt_{i+1}\wedge \dots \wedge dt_{k} \\
=&\int_{\sigma}dt_1\wedge\dots \wedge dt_{k} =  \frac{\sqrt{k+1}}{k! \sqrt{2^{k}}}
\end{align*}
since 
$$dt_i \wedge\dots\wedge dt_{i-1}\wedge dt_{i+1}\wedge \dots \wedge dt_{k} = (-1)^{i-1}dt_1 \wedge\dots\wedge dt_{i-1}\wedge dt_i \wedge dt_{i+1}\wedge \dots \wedge dt_{k}.$$
Hence we have
$$
 \mathscr I \circ {\mathscr W}(\chi_\sigma) = \frac{\sqrt{k+1}}{\sqrt{2^{k}}} \chi_\sigma
$$
and so for $c \in C^k_p(K)$ the following holds
$$
 \mathscr I \circ {\mathscr W}(c) = \frac{\sqrt{k+1}}{\sqrt{2^{k}}} c.
$$
Now we can denote 
$$
\tilde {\mathscr W} = \frac{\sqrt{2^{k}}}{\sqrt{k+1}} {\mathscr W}
$$
that implies that the following diagram commutes
$$
\xymatrix{C^k_p(K)\ar[dr]_{{\rm Id}_{C^k_p(K)}}\ar[r]^{\tilde{\mathscr W}}&S\mathscr{L}^k_p(K)\ar[d]^{\mathscr I}\\
&C^k_p(K)}
$$
and $\mathscr I$ is a retraction of morphism $\tilde{\mathscr W}$.
\end{proof}
\begin{lemma}
In the following diagram in ${\sf Vec}_{\mathbb R}$
$$
\xymatrix{
{\rm im}(d^{k})\arrow[r] \arrow@/{}^{2pc}/[rr]^0&S\mathscr{L}^{k+1}_p(K)\arrow[r]^{d^{k+1}}& S\mathscr{L}^{k+2}_p(K)\\
&{\rm ker}(\mathscr I)\arrow[u]\arrow@{-->}[lu].&
}
$$
the map $\ker(\mathscr I)$ factors through  ${\rm im}(d^k)$.
\end{lemma}
\begin{proof}
Assume that $\omega \in {\rm ker}(\mathscr I)$ is $k$-form. 
Let $\sigma$ be a $k$-dimensional simplex. Consider a complex
$$
\xymatrix{0\arrow[r]&S\mathscr{L}_p^0(\sigma,\partial \sigma)\arrow[r]^-d&\dots\arrow[r]^-d&S\mathscr{L}_p^{k-1}(\sigma,\partial \sigma)\arrow[r]^-d&S\mathscr{L}_p^{k}(\sigma,\partial \sigma)\arrow[r]&0}
$$
It is known that
$$
H^i(S\mathscr{L}^*_p(\sigma,\,\partial \sigma)) = \begin{cases}
\mathbb{R} \text{~for~} i=k,\\
0 \text{~for~} i < k.
\end{cases}
$$
in particular, we can say that the cohomology group $H^k(\Omega^*_p(\sigma,\,\partial \sigma))$ consists of elements of the following type $[c\omega]$ where $\omega$ is a cocycle such that 
$\int_{\sigma}\omega \ne 0$  and $c \in \mathbb R$. To put it otherwise, every closed form $\theta$ can be presented as below 
$$
\theta  = c \omega + d\eta,~ c = \frac {\int_{\sigma}\theta}{\int_{\sigma} \omega}.
$$ 
Hence  we can say that  ${\rm im}(d^{k-1})$ is precisely a set  of forms which have zero integral over $\sigma$. 
Since those forms constitute a closed set, so do elements of ${\rm im}(d^{k-1})$. Thereby  the following map
$$
d \colon S\mathscr{L}^{k-1}_p(\sigma,\,\partial\sigma) \to {\rm im}(d^{k-1})
$$
is an epimorphism in the category  ${\sf Ban}_{\infty}$
and  by the  Banach inverse operator theorem
there exists a constant $C$
such that for $d \colon \eta \mapsto {\rm res_{\sigma,\,K}}\omega$ we have the estimation  $\|\eta\|_{\Omega_{\infty}}\le C\|{\rm res_{\sigma,\,K}}\omega\|_{\Omega_{\infty}}$.

To summarize, for every $\sigma \in K[k]$ there exists $\eta_{\sigma}\in S\mathscr{L}^{k-1}_p(\sigma,\,\partial \sigma)$ such that ${\rm res}_{\sigma,\, K}\omega = d \eta_\sigma$.
Moreover, for the set of forms $\{\eta_\sigma\}_{\sigma \in K[k]}$ holds \linebreak
${\rm res}_{\sigma' \cap \sigma,\, \sigma}\eta_\sigma = {\rm res}_{\sigma' \cap \sigma,\, \sigma'}\eta_{\sigma'} = 0$ for any couple 
of simplexes $\sigma,~\sigma' \in K[k]$.  As a result we have a ${\rm res}_{K[k],\, K}\omega = d \eta$  where ${\rm res}_{\sigma,\, K[k]} \eta = \eta_\sigma$   and
$$
 \|\eta\|^p_{S\mathscr{L}^*_p(K[k])} = \sum_{\sigma \in K[k]}\|\eta_\sigma\|^p_{\Omega^*_{\infty}(\sigma)}
 \le  C\sum_{\sigma \in K[k]}\| {\rm res}_{\sigma,\, K}\omega \|^p_{\Omega^*_{\infty}(\sigma)} \le C\|\omega\|^p_{S\mathscr{L}^k_p(K)}.
 $$
Then we have a bounded map 
$$
\gamma \colon S\mathscr{L}^k_p(K) \to S\mathscr{L}^{k-1}_p(K[k],K[k-1]) 
$$
It is known that for every simplex $\delta \in K[{k+1}]$ and every form $\alpha \in S\mathscr{L}_{p}^k(\partial \delta)$ 
there is a morphism of normed spaces $s \colon S\mathscr{L}_{p}^k(\partial \delta)\to S\mathscr{L}_{p}^k(\delta)$ continuing forms off the boundary of simplex to its interior.

In the light of previous steps, it was in fact established that 
$$
s\gamma \colon S\mathscr{L}^k_p(K) \to S\mathscr{L}^{k-1}_p(K[k+1]) 
$$
is a bounded map.
Let us look at 
$
\omega'  =   {\rm res}_{K[{k+1}],\, K}\omega - d(s\gamma \omega).
$
It is clear that $\omega'$ is a closed form, as well as $\omega$, and $\omega' \in S\mathscr{L}^{k}_{p}(K[{k+1}],\,K[k])$. The restriction of $\omega'$ to every simplex $\sigma \in K[{k+1}]$ is an exact form
since $H^{k}(S\mathscr{L}^*_{p}(\sigma,\,\partial \sigma))  = 0$.  In other words ${\rm im}(d^k) = {\rm ker}(d^{k+1})$
and 
$$
d \colon S\mathscr{L}^{k-1}_p(\sigma,\,\partial\sigma) \to {\rm Im}(d^{k-1})
$$
is an epimorphism in the category  ${\sf Ban}_{\infty}$.
Using  the Banach inverse operator theorem and the above argument
we can see that there exists 
a bounded map 
$$
\gamma^1 \colon S\mathscr{L}^k_p(K) \to S\mathscr{L}^{k-1}_p(K[k+1],K[k]) 
$$
Then we can take
$$
s\gamma^1 \colon S\mathscr{L}^k_p(K) \to S\mathscr{L}^{k-1}_p(K[k+2]) 
$$
and so on. In essence, repeating this construction for every dimension as a result we obtain a finite composition of bounded operators which can be presented as follows
$$
 {\rm ker}(\mathscr I) \to S\mathscr{L}^{k-1}_p(K).
$$
And moreover, the map $\ker(\mathscr I)$ factors through  ${\rm im}(d^k)$.
\end{proof}
\begin{theorem}
Assume $K$ is a simplicial complex of bounded geometry then the following cohomology groups are isomorphic in the category ${\sf Vec}_{\mathbb R}$:
$$
H^k(C^*_p(K)) \cong H^k(S\mathscr{L}^*_p(K)).
$$
\end{theorem}
\begin{proof}
\begin{align*}
H^k(S\mathscr{L}^*_p(K)) &\cong H^k(C^*_p(K) \oplus \ker^* \mathscr I) \\
\cong  H^k(C^*_p(K)) \oplus& H^k(\ker ^*\mathscr I) \cong H^k(C^*_p(K))
\end{align*}
\end{proof}
\begin{appendices}
\section{Appendix}

It is not hard to see that obtained results hold for  metric simplicial complexes of bounded geometry  with arbitrary $L\ge 1$. 
We should notice that such complexes emerge naturally as a triangulation of Riemannian  manifolds with bounded geometry (see, for example, \cite{OA}, \cite{DK}):
\begin{theorem}
Let $M$ be an $n$-dimenthional Riemannian manifold of bounded geometry with geometric bounds $a$, $b$, $\epsilon$. Then 
$M$ admits a triangulation $K$ of bounded geometry (whose geometric bounds depend on $n$, $a$, $b$, $\epsilon$) and an $L$-bi-Lipcshitz homeomorphism $f\colon K \to M$,
where $L= L(n,\,a,\,b,\,\epsilon)$.
\end{theorem}

Also we can generalize our results in the following way. 
\begin{definition}
We will call an $n$-bouquet of $m$ balls  a  subset of $\mathbb R^{n+1}$ obtained as a union of $m$ $n$-balls of radius $1$ with a common centre.
\end{definition}
\begin{lemma}
Let $R$ be an $n$-bouquet.   
Then any $ \omega \in \Omega^k_{p,\,p}(R)$ can be extended to an $n+1$-ball $B_1$ of radius $1$ in such a way that
$ \tilde{\omega} \in \Omega_{p,\,p}(B_1)$ and $\|\tilde{\omega}\|_{\Omega^{*}_{p,\,p}(B_1)} \le \|\omega\|_{\Omega^{*}_{p,\,p}(R)}$. 
\end{lemma}
\begin{proof}
Assume that $ \omega \in \Omega^k_{p,\,p}(R)$. Due to the definition of an $n$-bouquet $B_1\setminus R$ is a number of disconnected sets which are contractible to a point.
$$
B_1\setminus R = \coprod N_i
$$
Fix some $i$ and let $M = R\cap \overline{N}_i$. So $M\times[0,\,1]$ is homeomorphic to $\overline{N}_i$ and  $M$ is a retract of $N_i$.
Then using the argument of {\it Lemma} 4 we can define a bounded map 
$$
\rho \colon\Omega^{k}_{p,\,p}(M) \to \Omega^{k}_{p,\,p}(M\times[0,\,1])
$$
as follows
$$
\omega \mapsto \tilde{\omega} = (1-s)\omega.  
$$ 
\end{proof}

\begin{theorem}
Let $K$ be a complex of bounded geometry and
a star $\Sigma_i$ of $K$ at a vertex $e_i$  is homeomorphic to
an n-bouquets. 
Then there exist operators $\mathcal R_i$ and $\mathcal A_i$ such that
the diagram
$$
\xymatrix{ \dots\ar[r]^-{d}&\Omega^{k-1}_{p,\,p}({ K}) \ar[rr]^-{d}\ar[d]^-{\mathcal{R}_i}
&&\Omega^{k}_{p,\,p}({ K})\ar[rr]^-{d}\ar[lldd]_{\mathcal{A}_i}\ar[d]^-{\mathcal{R}_i}
&&\Omega^{k+1}_{p,\,p}({ K})\ar[lldd]_{\mathcal{A}_i}\ar[r]^-{d}\ar[d]^-{\mathcal{R}_i}&\dots\\
\dots &S\mathscr{L}^{k-1}_p (K)\ar@{_(->}[d]&&S\mathscr{L}^{k}_p(K)\ar@{_(->}[d]&&S\mathscr{L}^{k+1}_p(K)\ar@{_(->}[d]&\dots\\
\dots\ar[r]_-{d}&\Omega^{k-1}_{p,\,p}(K)\ar[rr]_-{d}&&\Omega^k_{p,\,p}(K)\ar[rr]_-{d}&& \Omega^{k+1}_{p,\,p}(K)\ar[r]_-{d}&\dots
}
$$
has commutative squares in the category of Banach spaces ${\sf Ban}_{\infty}$. Moreover the following holds
$$
\mathcal R_i  - \mathrm{Id}_{\Omega^*_{p,\,p}} = d\mathcal A_i + \mathcal A_i d
$$
\end{theorem}
\begin{proof}
Assume that $X \in T(M\times[0,\,1])$ then 
$$
X = a\frac{\partial}{\partial s} + V,~V \in T(M)
$$
and
$$
\iota_X((1-s)\omega) = (1-s)\omega \iota_X(\omega) = (1-s)\iota_V(\omega)
$$
\begin{align*}
\mathcal A_i \tilde{\omega} = \int\limits_{{\mathbb R}^{n+1}} \left(\int\limits^{1}_{0} \iota_{\frak X_{\varepsilon x}(\mathfrak s_{\varepsilon xt}(t))}
(\mathfrak s^{*}_{\varepsilon xt}\tilde{\omega}) dt\right) \cdot \tau(x)\\
= C(s)\int\limits_{{\mathbb R}^{n}} \left(\int\limits^{1}_{0} \iota_{\frak X_{\varepsilon v}(\mathfrak s_{\varepsilon xt}(t))}(\mathfrak s^{*}_{\varepsilon vt}{\omega}) dt\right) \cdot \tau(v)
\end{align*}
It follows that
$$
\|\mathcal A_i \tilde{\omega}\|_{\Omega_{p,\,p}(M\times[0,\,1])} = \| C(s)\mathcal A_i \omega \|_{\Omega_{p,\,p}(M\times[0,\,1])} = C\|\mathcal A_i \omega \|_{\Omega_{p,\,p}(M)}
$$
and so 
$$\|\mathcal A_i \tilde{\omega}\|_{\Omega_{p,\,p}(B_1)} = C \|\mathcal A _i\tilde{\omega}\|_{\Omega_{p,\,p}(R)}.$$
Let a star $\Sigma_i$ of complex $K$ is bi-Lipschitz homeomorphic to  an $n$-bouquet $R$
$$
\varphi \colon \Sigma_i \to R
$$
As above we have  a bounded map
$$
\rho \colon\Omega^{k}_{p,\,p}(R) \to \Omega^{k}_{p,\,p}(B_1)
$$
It was shown that there exists a bounded map
$$
\rho^{-1} \mathcal A_i \rho 
$$
moreover 
$$
\rho^{-1} \mathcal R_i \rho 
$$
is bounded as well since $\mathcal R_i \rho (\omega) \in S{\mathscr L}^{k}_p(B_1)$.
And as a result we can define 
$$
\varphi^{-1}\rho^{-1} \mathcal A_i \rho  \varphi\colon\Omega^{k}_{p,\,p}(\Sigma_i) \to \Omega^{k-1}_{p,\,p}(\Sigma_i)
$$
and
$$
\varphi^{-1}\rho^{-1} \mathcal R_i \rho  \varphi\colon\Omega^{k}_{p,\,p}(\Sigma_i) \to S{\mathscr L}^{k}_p(\Sigma_i).
$$

In the light of the above de Rham regularization preserves all established properties for simlicial complexes with stars which are homeomorphic to  $n$-bouquets. 
\end{proof}
\end{appendices}

\end{document}